\documentclass[11pt,draftclsnofoot,onecolumn]{IEEEtran}

\usepackage{mydef}
\usepackage{enumerate}

\usepackage[normalem]{ulem}
\usepackage{xcolor}
\usepackage{hyperref}
\usepackage{array}

\makeatletter
\newcommand{\thickhline}{%
    \noalign {\ifnum 0=`}\fi \hrule height 1pt
    \futurelet \reserved@a \@xhline
}
\newcolumntype{"}{@{\hskip\tabcolsep\vrule width 1pt\hskip\tabcolsep}}
\makeatother

\usepackage[utf8]{inputenc} 
\usepackage[T1]{fontenc}    
\usepackage{hyperref}       
\usepackage{url}            
\usepackage{booktabs}       
\usepackage{amsfonts}       
\usepackage{nicefrac}       
\usepackage{microtype}      
\usepackage{cases}
\pdfoutput=1
\begin{document}
\title{Decentralized Proximal Gradient Algorithms with Linear Convergence Rates}

\author{
Sulaiman A. Alghunaim, Ernest K. Ryu, Kun Yuan, and     Ali H.~Sayed, ~\IEEEmembership{Fellow,~IEEE}
\thanks{ S. A. Alghunaim is with the Department of Electrical Engineering, Kuwait University (e-mail: {\tt\small salghunaim@ucla.edu}). K. Yuan is with Alibaba Group (US), Bellevue, WA, USA (e-mail: {\tt\small kunyuan@ucla.edu}). This work was done while they were Ph.D. students at UCLA. E. K. Ryu is with the Department of Mathematical Sciences, Seoul National University (e-mail:
{\tt\small ernestryu@snu.ac.kr}).  A. H. Sayed is with the Ecole Polytechnique Federale de Lausanne
(EPFL), School of Engineering (e-mail:
{\tt\small ali.sayed@epfl.ch}).}}

\markboth{}%
{Shell \MakeLowercase{\textit{et al.}}: Bare Demo of IEEEtran.cls for Journals}

\maketitle
\date{}

\begin{abstract}
This work studies a class of non-smooth decentralized multi-agent optimization problems where the agents aim at minimizing a sum of local strongly-convex smooth components plus a common non-smooth term. We propose a general primal-dual algorithmic framework that unifies many existing state-of-the-art algorithms. We establish linear convergence of the proposed method to the exact solution  in the presence of the non-smooth term.
  Moreover, for the more general class of problems with agent specific non-smooth terms,  we  show that linear convergence cannot be achieved (in the worst case) for the class of algorithms that uses the gradients and the proximal mappings of the smooth and non-smooth parts, respectively. We further provide a numerical  counterexample that shows how some state-of-the-art algorithms   fail to converge linearly  for strongly-convex objectives and different local non-smooth terms. 
  \end{abstract}
\begin{IEEEkeywords}
Decentralized optimization, proximal gradient algorithms, linear convergence, gradient tracking, diffusion, unified decentralized algorithm.
\end{IEEEkeywords}

\section{Introduction}
In this work, we consider a static and undirected network of $K$ agents  connected over some  graph where each agent $k$ owns a private cost function $J_k: \real^{M} \rightarrow \real$.  Through only local interactions (i.e., with agents only communicating  with their immediate neighbors), each agent is interested in finding a solution to the following problem: 
\begin{align}
w^\star \in \argmin_{w\in \mathbb{R}^M} \quad
  \frac{1}{K}\sum_{k=1}^K J_k(w) + R(w) \label{decentralized1} 
\end{align}
 where  $R:\real^{M} \rightarrow \real \cup \{+ \infty \}$  is a convex function  (not necessarily differentiable).    We adopt the following assumption throughout this work.
\begin{assumption} \label{assump:cost}
{\rm ({\bf Cost function}):  We assume that a solution exists to problem \eqref{decentralized1} and each cost function $ J_k(w)$ is first-order differentiable and  $\nu$-strongly-convex:
\eq{
(w^o-w^\bullet)\tran \big(\grad J_k(w^o)-\grad J_k(w^\bullet)\big) &\geq \nu \|w^o-w^\bullet\|^2  \label{stron-convexity} 
} 
with $\delta$-Lipschitz continuous gradients:
\eq{
\|\grad J_k(w^o)-\grad J_k(w^\bullet)\| &\leq \delta \|w^o-w^\bullet\|  \label{lipschitz}
}
\noindent for any $w^o$ and $w^\bullet$. Constants $\nu$ and $\delta$ are strictly positive and  satisfy $\nu\leq \delta$. 
We also assume $R(w)$ to be a proper\footnote{The function $f(.)$ is proper if $-\infty <f(x)$ for all $x$ in its domain and $f(x)< \infty$ for at least one $x$.}  and lower-semicontinuous convex function.
\qd
}
\end{assumption}
Note that from the strong-convexity condition \eqref{stron-convexity}, we know the objective function in \eqref{decentralized1} is also strongly convex and, thus, the global solution $w^\star$ is unique.

\subsection{Related Works}
 Various algorithms have been proposed to solve decentralized optimization problems of the form \eqref{decentralized1} -- see  \cite{
 xu2015augmented,nedic2017geometrically,di2016next,
 scutari2019distributed,sun2016distributed, qu2017harnessing,
 nedic2017achieving,shi2015extra,ling2015dlm,yuan2019exactdiffI, li2017nids,pu2018push,xin2018linear,sun2019convergence}.  Only few works have attempted to unify some of these various algorithms  \cite{jakovetic2019unification,sundararajan2018canonical,
 sundararajan2019analysis}. For example, the work \cite{jakovetic2019unification} proposed a general method that includes EXTRA \cite{shi2015extra}  and DIGing \cite{nedic2017achieving} (for static and undirected network)  as special cases. However, the  method in \cite{jakovetic2019unification} does not include the adapt-then-combine\footnote{The Adapt-then-Combine (ATC) structure was proposed in  \cite{cattivelli2010diffusion} to distinguish between different implementations of diffusion learning strategies -- see also  \cite[Ch. 7]{sayed2014nowbook}. } (ATC) gradient-tracking algorithms \cite{xu2015augmented,nedic2017geometrically,scutari2019distributed}.  The work \cite{sundararajan2018canonical} proposed a canonical form that characterizes decentralized algorithms that require a single round of communication and gradient computation per iteration, which  does not include the Aug-DGM (ATC-DIGing) \cite{xu2015augmented,nedic2017geometrically}. Reference \cite{sundararajan2018canonical}   only focused on the canonical form without focusing on the analysis of this form.   Later, the work \cite{sundararajan2019analysis}  studied a class of the canonical form in \cite{sundararajan2018canonical} over time-varying connected networks and provided worst case linear convergence rates through numerical solution of semidefinite programs.

 Different from \cite{jakovetic2019unification,sundararajan2018canonical,
 sundararajan2019analysis} we propose a general primal-dual framework that unifies  many existing algorithms including EXTRA \cite{shi2015extra}, DLM \cite{ling2015dlm}, Exact diffusion \cite{yuan2019exactdiffI}, NIDS \cite{li2017nids}, and different implementations of the gradient tracking methods \cite{xu2015augmented,nedic2017geometrically,scutari2019distributed,
 di2016next,qu2017harnessing} including Aug-DGM \cite{xu2015augmented}. Our  framework shows that the ATC gradient-tracking methods can  be represented as primal-dual recursions.    The work \cite{alghunaim2019linearly} proposed a proximal gradient algorithm that solves \eqref{decentralized1} and established its linear convergence to the {\em exact} solution $w^\star$.  Motivated by the technique from \cite{alghunaim2019linearly}  we extend the proposed general framework to handle the non-smooth term $R(w)$ and prove linear convergence of the proposed general method  to the solution $w^\star$ in the presence of the non-smooth term.

In order to establish global linear convergence, this work  considers the non-smooth term to be common across all agents. One might wonder whether it is  possible for a  decentralized proximal gradient algorithm to achieve  global linear convergence in the presence of different local non-smooth  $R_k(w)$ terms. As far as we know, this question has not been explicitly answered in the literature.  Many decentralized optimization problems  where each agent $k$ has a local non-smooth term $R_k(w)$ possibly different from other agents \cite{li2017nids,shi2015proximal,chang2015multi,
aybat2018distributed,bianchi2015coordinate} have been proposed. None of these methods have been shown to achieve global linear convergence in the presence of general non-smooth terms.  By adjusting  the results from \cite{woodworth2016tight} to the decentralized optimization   set-up with agent-specific non-smooth terms $\{R_k(w)\}$, it can be shown that it is {\em impossible} for any proximal gradient based algorithm to achieve linear convergence  in the {\em  worst case}  -- see Section \ref{sec:sublinearbound}. Note that the works \cite{arjevani2015communication,uribe2020dual} showed that global linear convergence is not possible for   non-smooth strongly-convex functions in the worst case for the class of algorithms limited to one communication round but unlimited in the amount of computation and access to the functions  per iteration. In contrast, we consider algorithms unlimited in the number of communications rounds but limited to one gradient and proximal computations per iteration. We remark that under a common non-smooth term, the work \cite{sun2019convergence} also established global linear convergence for a decentralized algorithm that is based on successive convex approximation, which is different from our proximal primal-dual approach.

\subsection{Contribution}
  Given the above, this paper  has three contributions. 
  First, when $R(w)=0$ we propose a novel  primal-dual {\em unified decentralized algorithm} (UDA) that unifies many existing state-of-the-art algorithms including the ATC algorithms \cite{xu2015augmented,nedic2017geometrically,scutari2019distributed,
 di2016next,yuan2019exactdiffI, li2017nids} and the non-ATC algorithms \cite{qu2017harnessing,
 nedic2017achieving, shi2015extra,ling2015dlm}. To our knowledge, this is the first primal-dual interpretation of the  ATC gradient-tracking methods \cite{xu2015augmented,nedic2017geometrically,scutari2019distributed,
 di2016next}.  Second, we extend this framework to handle a {\em common} non-smooth regularization term  and provide a unifying linear convergence analysis  under proper conditions. Our step-size and convergence rate upper bounds shed light on the stability and performance of these various methods.      Third, by tailoring a result from \cite{woodworth2016tight}, we show that if each agent owns a non-smooth term, then linear convergence to the {\em exact} solution $w^\star$ cannot be achieved in the worst case for the class of decentralized algorithms  where each agent can compute one gradient and one proximal mapping per iteration for the smooth and non-smooth parts, respectively. We further provide a numerical counterexample where PG-EXTRA \cite{shi2015proximal} and proximal linearized ADMM \cite{chang2015multi,aybat2018distributed} fail to achieve global linear convergence for strongly-convex objectives.  

\subsection{Notation}
  For a vector $x \in \real^M$ and  a positive semi-definite matrix $C \geq 0$, we let $\|x\|_C^2=x\tran C x$. For any matrix $A$, we let $\sigma_{\max}(A)$ denote the maximum singular value of $A$ and $\underline{\sigma}(A)$ denote the minimum {\em non-zero} singular value of $A$. Moreover, for any symmetric matrices $A$ and $B$ with the same dimension, we let $A\ge B$ ($A>B$) if $A-B$ is positive semi-definite (positive definite). The $N \times N$ identity matrix is denoted by $I_N$. We let $\one_{N}$ be a vector of size $N$ with all entries equal to one. The Kronecker product is denoted by $\otimes$. We let ${\rm col}\{x_n\}_{n=1}^N$ denote a column vector (matrix) that stacks the vector (matrices) $x_n$ of appropriate dimensions on top of each other. The subdifferential $\partial f(x)$ of a function $f:\real^{M} \rightarrow \real$ at some $x \in \real^{M}$ is the set of all subgradients $
\partial f(x) = \{g \ | \ g\tran(y-x)\leq f(y)-f(x), \forall \ y \in \real^{M}\} $.
The proximal operator with parameter $\mu>0$ of a function $f:\real^{M} \rightarrow \real$ is
\eq{
{\rm \bf prox}_{\mu f}(x) = \argmin_z \ f(z)+{1 \over 2 \mu} \|z-x\|^2  \label{def_proximal}
}
\section{Unified Decentralized Algorithm (UDA)} \label{sec:ATC:smooth}
In this section, we present the {\em unified decentralized algorithm} (UDA) that covers various state-of-the-art algorithms as special cases. To this end, we will first focus on the smooth case ($R(w)=0$), which will then be extended to handle the non-smooth component $R(w)$ in the following section.
\subsection{General Primal-Dual Framework}
 For algorithm derivation and motivation purposes, we will rewrite problem \eqref{decentralized1} in an equivalent manner. To do that, we let  $w_k \in \real^M$ denote a local copy of $w$ available at agent $k$ and introduce the network quantities: 
 \eq{
 \sw \define {\rm col}\{w_1,\cdots,w_K\} \in \real^{KM}, \quad \cJ(\sw) &\define \frac{1}{K} \sum_{k=1}^K J_k(w_k)
 }   
 Further, we introduce two general symmetric matrices  $\cB \in \real^{MK \times MK}$ and  $\cC \in \real^{MK \times MK}$ that satisfy the following conditions:  
 \begin{subnumcases}{\label{consensus-condition-both}} 
	 \cB \sw=0 \iff w_1=\cdots=w_K \label{consensus-condition-B} \\
\cC=0 \quad {\rm or} \quad 	\cC \sw =0 \iff \cB \sw=0   \label{consensus-condition-C}	 
	\end{subnumcases}  
 For algorithm derivation, the matrices $\{\cB,\cC\}$ can be any general consensus matrices \cite{loizou2016new}. Later, we will see how to choose these matrices to recover different decentralized implementations -- see Section \ref{sec:specific_ins}.  With these quantities, it is easy to see that problem \eqref{decentralized1} with $R(w)=0$ is equivalent to the following problem:
\begin{align}
 \underset{\ssw\in \mathbb{R}^{KM}}{\text{minimize   }}& \quad
   \cJ(\sw)+\frac{1}{2 \mu}\| \sw\|_{\cC}^2 , \quad {\rm s.t.} \ \cB \sw=0\label{decentralized2} 
\end{align}
where $\mu >0$ and the matrix $\cC \in \real^{MK \times MK}$ is a positive semi-definite consensus penalty matrix satisfying \eqref{consensus-condition-C}.  To solve problem \eqref{decentralized2}, we consider  the saddle-point formulation:
\eq{
 \min_{\ssw} \max_{\ssy} \quad \cL(\sw,\sy)  \define \cJ(\sw)  + \frac{1}{ \mu} \sy\tran \cB\sw + \frac{1}{2 \mu}\| \sw\|_{\cC}^2
\label{saddle_point}
}
where $\sy \in \real^{MK}$ is the dual variable.  To solve \eqref{saddle_point}, we propose the following algorithm: let $\sy_{-1}=0$ and $\sw_{-1}$ take any arbitrary value. Repeat for $i=0,1,\cdots$

\begin{subnumcases}{\label{alg_ATC_framework}}
\ssz_i =  (I-\cC) \sw_{i-1}-\mu \grad \cJ(\sw_{i-1})  -  \cB \sy_{i-1} \label{z_ATC_DIG}  &\textbf{(primal-descent)} \\
\sy_i = \sy_{i-1}+ \cB  \ssz_i \label{dual_ATC_DIG}  &\textbf{(dual-ascent)} \\
\sw_i = \bar{\cA} \ssz_{i} \label{primal_ATC_DIG}  &\textbf{(Combine)} 
 \end{subnumcases}
 where $\bar{\cA}=\bar{A} \otimes I_M$ and  $\bar{A}$ is a symmetric and doubly-stochastic combination matrix.  In the above UDA algorithm, step \eqref{z_ATC_DIG} is a gradient descent followed by a gradient ascent step in  \eqref{dual_ATC_DIG}, both applied to the saddle-point problem \eqref{saddle_point} with step-size $\mu$.  The last step \eqref{primal_ATC_DIG} is a combination step that enforces further agreement. Next we show that by proper choices of $\bar{\cA}$, $\cB$, and $\cC$ we can recover many state of the art algorithms. To do that,  we need to introduce the  combination matrix  associated with the  network.
\subsection{Network Combination Matrix} \label{sec:combina:matrix}
   Thus, we introduce the combination matrices
  \eq{
  A=[a_{sk}] \in \real^{K \times K}, \quad  \cA= A \otimes I_M   \label{combination-cal-A}
}  
  where the entry  $a_{sk}=0$ if there is no edge connecting agents $k$ and $s$. The matrix $A$ is assumed to be  symmetric and doubly stochastic matrix (different from $\bar{A}$).  We further assume the matrix to be  primitive, i.e., there exists an integer $j>0$ such that all entries of  $A^j$ are positive.  
 Under these conditions it holds that $(I_{MK}-\cA) \sw=0$ if and only if $w_k=w_s$ for all $k,s$ --- see \cite{shi2015extra,yuan2019exactdiffI}. 
\subsection{Specific Instances} \label{sec:specific_ins}
We start by rewriting  recursion \eqref{alg_ATC_framework} in an equivalent manner by eliminating the dual variable $\sy_i$. Thus, from \eqref{z_ATC_DIG} it holds that
 \eq{
\ssz_i-\ssz_{i-1} &= (I-\cC) (\sw_{i-1}-\sw_{i-2})-  \cB (\sy_{i-1}-\sy_{i-2}) -\mu \big(\grad \cJ(\sw_{i-1})-\grad \cJ(\sw_{i-2})\big)  \nonumber \\
 &\overset{\eqref{dual_ATC_DIG}}{=}  (I-\cC) (\sw_{i-1}-\sw_{i-2})-  \cB^2   \ssz_{i-1} -\mu \big(\grad \cJ(\sw_{i-1})-\grad \cJ(\sw_{i-2})\big)   \nonumber
}
Rearranging the previous equation  we get:
 \eq{
\ssz_i &= (I-\cB^2) \ssz_{i-1} + (I-\cC) (\sw_{i-1}-\sw_{i-2}) -\mu \big(\grad \cJ(\sw_{i-1})-\grad \cJ(\sw_{i-2})\big) 
\label{eq:sub_atc}
}
  Utilizing this property, we will now choose specific matrices $\{\bar{\cA},\cB,\cC\}$ and show that we can recover many state of the art algorithms (see Table \ref{table}): 
\subsubsection{\bf Exact diffusion \cite{yuan2019exactdiffI}}
  If we choose $\bar{\cA}=0.5 (I+\cA)$, $\cC=0$ and $\cB^2=0.5 (I- \cA)$ in \eqref{eq:sub_atc}, we get: 
 \eq{
\ssz_i &= 
\bar{\cA}
 \ssz_{i-1} +  \sw_{i-1}-\sw_{i-2} -\mu \big(\grad \cJ(\sw_{i-1})-\grad \cJ(\sw_{i-2})\big) 
}
 Multiplying the previous equation by $\bar{\cA}$ and noting from \eqref{primal_ATC_DIG} that $\sw_i= \bar{\cA} \ssz_{i}$, we get:
 \eq{
\sw_i=\bar{\cA} \bigg( 2 \sw_{i-1}
 -   \sw_{i-2} -\mu \big(\grad \cJ(\sw_{i-1})-\grad \cJ(\sw_{i-2})\big)  \bigg)  \label{exact-diffusion}
}
  The above recursion is the exact diffusion recursion first proposed in \cite{yuan2019exactdiffI}. We also note that if we choose $\cC=0$, $\cB^2=c (I- \cA)$ ($c \in \real$), and $\bar{\cA}=I-\cB^2$ then we recover the smooth case of the NIDS algorithm from \cite{li2017nids}. As highlighted in \cite{li2017nids},  NIDS is identical to exact diffusion for the smooth case  when $c=0.5$.
\subsubsection{\bf Aug-DGM  \cite{xu2015augmented}}
 Let $\cC=0$, $\bar{\cA}=\cA^2$, and $\cB=I-\cA$. Substituting into \eqref{eq:sub_atc}:
 \eq{
\ssz_i &= (2\cA-\cA^2) \ssz_{i-1} +  \sw_{i-1}-\sw_{i-2} -\mu \big(\grad \cJ(\sw_{i-1})-\grad \cJ(\sw_{i-2})\big) 
}
By multiplying the previous equation by $\bar{\cA}=\cA^2$ and noting from \eqref{primal_ATC_DIG} that $\sw_i= \cA^2 \ssz_{i}$, we get the recursion:
\eq{
\sw_i=\cA \bigg( 2 \sw_{i-1}
 - \cA  \sw_{i-2} -\mu \cA \big(\grad \cJ(\sw_{i-1})-\grad \cJ(\sw_{i-2})\big)  \bigg) \label{atc_DGM_eliminate} 
}
The above recursion is equivalent to the Aug-DGM \cite{xu2015augmented} (also known as ATC-DIGing \cite{nedic2017geometrically}) algorithm:
\begin{subequations} \label{atc_DGM}
\eq{ 
\sw_i&=\cA(\sw_{i-1}-\mu  \ssx_{i-1}) \label{atc-dgm1} \\
\ssx_{i}&=\cA \big(\ssx_{i-1}+ \grad \cJ(\sw_i)- \grad \cJ(\sw_{i-1}) \big) \label{atc-dgm2}
}
\end{subequations}
By eliminating the gradient tracking variable $\ssx_{i}$, we can rewrite the previous recursion as \eqref{atc_DGM_eliminate} -- see Appendix \ref{supp_equiva_representation}.
\subsubsection{\bf ATC tracking method \cite{di2016next,scutari2019distributed}}
Let $\cC=I-\cA$ and $\cB=I-\cA$. Substituting into \eqref{eq:sub_atc}:
 \eq{
\ssz_i &=(2\cA  -\cA^2) \ssz_{i-1}  
+ \cA \sw_{i-1} - \cA  \sw_{i-2} -\mu \big(\grad \cJ(\sw_{i-1})-\grad \cJ(\sw_{i-2})\big) 
}
By multiplying the previous equation by $\bar{\cA}=\cA$ and noting from \eqref{primal_ATC_DIG} that $\sw_i= \cA \ssz_{i}$, we get the recursion:
\eq{
\sw_i=\cA \bigg( 2 \sw_{i-1}
 - \cA  \sw_{i-2} -\mu \big(\grad \cJ(\sw_{i-1})-\grad \cJ(\sw_{i-2})\big)  \bigg)  \label{next_eliminate}
}
The above recursion is equivalent to the following variant of the ATC tracking method  \cite{di2016next,scutari2019distributed}:
\begin{subequations} \label{next}
\eq{
\sw_i&=\cA(\sw_{i-1}-\mu  \ssx_{i-1}) \label{next1} \\
\ssx_{i}&=\cA \ssx_{i-1}+ \grad \cJ(\sw_i)- \grad \cJ(\sw_{i-1}) \label{next2}
}
\end{subequations}
By eliminating the gradient tracking variable $\ssx_i$,  we can show that the previous recursion is exactly \eqref{next_eliminate} -- see Appendix \ref{supp_equiva_representation}. 
\subsubsection{\bf NON-ATC Algorithms ($\bar{\cA}=I$)}
We note that DIGing \cite{qu2017harnessing,nedic2017achieving}, EXTRA \cite{shi2015extra}, and the decentralized linearized alternating direction method of multipliers (DLM) \cite{ling2015dlm} can also be represented by \eqref{alg_ATC_framework} with $\bar{\cA}=I$ and proper choices of $\cB^2$ and $\cC$ -- see Table \ref{table}.   Since  $\bar{\cA}=I$, these algorithms are not of the ATC form.  Please see Appendix \ref{supp_non_atc}  for the details and analysis of non-ATC case.
\begin{table}[t] 
\caption{Listing of some state-of-the-art first-order algorithms that can recovered by specific choices of $\bar{\cA}$, $\cB$, and $\cC$ in \eqref{alg_ATC_framework}. The matrix $\cA$ is a typical symmetric and doubly stochastic network combination matrix introduced in \eqref{combination-cal-A}. The matrix $\cL$ is chosen such that the $k$-th block of $\cL \sw_i$ is equal to $\sum_{s \in \cN_k} w_{k,i}-w_{s,i}$ and $c>0$ is a step-size parameter.  }
\centering
\large 
\begin{tabular}{|c|c|c|c|}
\thickhline
\rowcolor[HTML]{C0C0C0} 
{\bf ATC algorithms}      & $\bar{\cA}$      & $\cB^2$     & $\cC$     \\ \thickhline
\cellcolor[HTML]{EFEFEF} Aug-DGM/ATC-DIGing  \cite{xu2015augmented,nedic2017geometrically}                &  $\cA^2$      &      $(I-\cA)^2$ &    $0$   \\ \hline
 \cellcolor[HTML]{EFEFEF}   ATC tracking \cite{di2016next,scutari2019distributed}            &     $\cA$   &   $(I-\cA)^2$    &   $I-\cA$    \\ \hline
\cellcolor[HTML]{EFEFEF}     Exact diffusion \cite{yuan2019exactdiffI}           &       $0.5(I+\cA)$ &    $0.5(I-\cA)$   & 0      \\ \hline
 \cellcolor[HTML]{EFEFEF}     NIDS \cite{li2017nids}           &       $I-c(I-\cA)$ &    $c(I-\cA)$   & 0      \\ \thickhline
   \rowcolor[HTML]{C0C0C0} 
{\bf NON-ATC algorithms}      & $\bar{\cA}$      & $\cB^2$     & $\cC$     \\ \thickhline
 \cellcolor[HTML]{EFEFEF}  DIGing \cite{qu2017harnessing,nedic2017achieving}             &        $I$ &    $(I-\cA)^2$   & $I-\cA^2$      \\ \hline 
\cellcolor[HTML]{EFEFEF} EXTRA \cite{shi2015extra}                &    $I$    &    $0.5(I-\cA)$   &   $0.5(I-\cA)$  \\ \hline 
\cellcolor[HTML]{EFEFEF} DLM \cite{ling2015dlm}                &    $I$    &    $c \mu \cL$   &   $c \mu \cL$    \\ \thickhline
\end{tabular}
 \label{table}
\end{table}
\begin{remark} [\sc Communication cost] \label{remak:sharing-variable}{\rm
Note that exact diffusion \eqref{exact-diffusion} requires one round of communication or combination per iteration. This means that each agent sends an $M$ vector to its neighbor per iteration. On the other hand, the gradient tracking method \eqref{next} requires two rounds of combination/communication per iteration for the vectors $\sw_{i-1}-\mu  \ssx_{i-1}$ and $\ssx_{i-1}$, which means each agent sends a $2M$ vector to its neighbor. Similarly, the Aug-DGM (ATC-DIGing) method \eqref{atc_DGM} also requires two rounds of combination per iteration for the vectors  $\sw_{i-1}-\mu  \ssx_{i-1}$ and $\ssx_{i-1}+ \grad \cJ(\sw_i)- \grad \cJ(\sw_{i-1})$; moreover, it requires communicating these two variables sequentially (at different communication steps). \qd
}
\end{remark}

	\section{Proximal Unified Decentralized Algorithm (PUDA)}
   In this section, we extend UDA \eqref{alg_ATC_framework} to handle the non-differentiable component $R(w)$ to get a proximal unified decentralized algorithm (PUDA). Let us introduce the network quantity
\eq{
  \cR(\sw) &\define  {1 \over K} \sum_{k=1}^K R(w_k)  
} 
With this definition, we propose the following recursion:  let $\sy_{-1}=0$ and $\sw_{-1}$ take any arbitrary value. Repeat for $i=0,1,\ldots$
\begin{subnumcases}{ \label{alg_prox_ATC_framework}} 
\ssz_i =  (I-\cC) \sw_{i-1}-\mu \grad \cJ(\sw_{i-1})  -  \cB \sy_{i-1} \label{z_prox_ATC_DIG}   \\
\sy_i = \sy_{i-1}+ \cB  \ssz_i \label{dual_prox_ATC_DIG}   \\
\sw_i = {\rm \bf prox}_{\mu \cR}\big(\bar{\cA} \ssz_{i} \big) \label{primal_prox_ATC_DIG}  
 \end{subnumcases}
 We refer the reader to Appendix \ref{supp_equiva_represent_prox} for specific instances of PUDA \eqref{alg_prox_ATC_framework} and how to implement them in a decentralized manner.  In the following, we will show that $\sw_i$ in the above recursion converges to $\one_K \otimes w^\star$ where $w^\star$ is the desired solution of \eqref{decentralized1}.  We first prove the existence and optimality of the fixed points of recursion \eqref{alg_prox_ATC_framework}.
\begin{lemma}[\sc Optimality Point] \label{lemma:existence_fixed_optimality}{\rm Under Assumption \ref{assump:cost} and condition \eqref{consensus-condition-both}, a fixed point $(\sw^\star, \sy^\star, \ssz^\star)$ exists for recursions \eqref{z_prox_ATC_DIG}--\eqref{primal_prox_ATC_DIG}, i.e., it holds that
	\begin{subnumcases}{}
	\hspace{.5mm} \ssz^\star =\sw^\star-\mu \grad \cJ(\sw^\star)- \cB \sy^\star \label{p-d_ed-star} \\
	\hspace{2.8mm} 0 =  \cB \ssz^\star \label{d-a_ed-star} \\
	\sw^\star = {\rm \bf prox}_{\mu \cR}(\bar{\cA} \ssz^\star) \label{prox_step_ed-star}
	\end{subnumcases}
Moreover, $\sw^\star$ and $\ssz^\star$ are unique with $\sw^\star=\one_K \otimes w^\star$ where $w^\star$ is the solution of problem \eqref{decentralized1}.
	}
\end{lemma}
\begin{proof}   See Appendix \ref{supp_lemma_fixed}. 
\end{proof} 
 \section{Linear Convergence}
Note that  there exists a particular fixed point $(\sw^\star, \sy_b^\star, \sz^\star)$ where $\sy_b^\star$ is a unique vector that belongs to the range space of $\cB$ -- see \cite[Remark 2]{alghunaim2019linearly}. In the following we will show that the iterates $(\sw_i, \sy_i, \sz_i)$ converge linearly to this particular fixed point $(\sw^\star, \sy_b^\star, \sz^\star)$.  To this end, we introduce the error quantities:
\begin{align}
	\tsw_i\define \sw_i-\sw^\star, \quad \tsy_i \define \sy_i - \sy^\star_b, \quad \tsz_i \define \ssz_i-\ssz^\star
\end{align}
Note that from condition \eqref{consensus-condition-both} we have $\cC\sw^\star=0$. Therefore, from \eqref{z_prox_ATC_DIG}--\eqref{primal_prox_ATC_DIG} and \eqref{p-d_ed-star}--\eqref{prox_step_ed-star} we can reach the following error recursions:
\begin{subnumcases}{}
\tsz_i=(I-\cC)\tsw_{i-1}-\mu \big(\grad \cJ(\sw_{i-1})-\grad \cJ(\sw^\star) \big) - \cB \tsy_{i-1} \label{error_primal_ed} \\
\tsy_i = \tsy_{i-1}+ \cB \tsz_i \label{error_dual_ed} \\
\tsw_i = {\rm \bf prox}_{\mu \cR}\big(\bar{\cA}  \ssz_i\big)-{\rm \bf prox}_{\mu \cR}(\bar{\cA}  \ssz^\star) \label{error_prox_ed}
\end{subnumcases}
For our convergence result, we need the following technical conditions.
\begin{assumption}[\sc Consensus matrices] \label{assump_combination}
{\rm It is assumed that both condition \eqref{consensus-condition-both} and the following condition hold:
\eq{
\bar{\cA}^2 \leq I-\cB^2 \ {\rm and} \  0 \leq \cC < 2I \label{eq:asump_penalty} }
\qd
}
\end{assumption} 
\begin{remark}[\sc Convergence conditions]{\rm
\label{remark:conv_conditions}
Note that the above conditions are satisfied for  exact diffusion \cite{yuan2019exactdiffII} and NIDS \cite{li2017nids}. 
For the ATC tracking methods \eqref{atc_DGM} and \eqref{next}, the conditions translate to the requirement that the eigenvalues of $A$ are between  $[0,1]$, rather than the typical $(-1,1]$.
   Although this condition is not necessary, it can be easily satisfied by redefining $A \leftarrow 0.5 (I+A)$. We also impose it to unify the analysis of these methods through a short proof.   Note that most works that analyze decentralized methods under more relaxed conditions on the network topology  impose restrictive step-size conditions that depend on the network and on the order of $O(\nu^{\theta_1}/ \delta^{\theta_2})$ where $0 < \theta_1 \leq 1$ and $\theta_2>1$  -- see \cite{nedic2017geometrically,qu2017harnessing,pu2018push,
jakovetic2019unification}.   On the other hand, we require step sizes of  order  $O(1 / \delta)$. Moreover,  we will show that any algorithm that fits into our setup with $\cC=0$ can use a step-size as large as the centralized proximal gradient descent -- see discussion after Theorem \ref{theorem_lin_convergence}.    \qd
}
\end{remark}
Note that $\cB^2$ and $\cC$ are symmetric; thus, their singular values are equal to their eigenvalues. Moreover, since the square of a symmetric matrix is positive semi-definite, Assumption \ref{assump_combination} implies $0 < \underline{\sigma}(\cB^2) \leq 1$ and $\sigma_{\max}(\cC)<2$.

\begin{theorem}[\sc Linear Convergence]\label{theorem_lin_convergence}
{\rm	Under Assumptions \ref{assump:cost}--\ref{assump_combination}, if $\sy_0=0$ and the step-size satisfies  \eq{
\mu < {2-\sigma_{\max}(\cC) \over \delta},
}
 it holds that
\eq{
	\|\tsw_i\|^2+ \|\tsy_i\|^2 
	&\leq \gamma \big(\|\tsw_{i-1}\|^2+ \|\tsy_{i-1}\|^2 \big)
}
where   $\gamma= \max \big\{ 1 \hspace{-0.5mm}-\hspace{-0.5mm} \mu \nu (2-\sigma_{\max}(\cC)-\mu \delta ) ,1 - \underline{\sigma}(\cB^2) \big\}<1$.}
\end{theorem}
\begin{proof}  Squaring both sides of \eqref{error_primal_ed} and \eqref{error_dual_ed} we get
\eq{
	\|\tsz_i\|^2&= \|(I-\cC)\tsw_{i-1}-\mu \big(\grad \cJ(\sw_{i-1})-\grad \cJ(\sw^\star) \big)\|^2 +  \| \cB  \tsy_{i-1}\|^2 \nonumber \\
	& \ -2   \tsy_{i-1}\tran \cB \left((I-\cC)\tsw_{i-1}-\mu \big(\grad \cJ(\sw_{i-1})-\grad \cJ(\sw^\star) \big)\right) 
	\label{er_sq_primal_ed}
}
and
\eq{
	\|\tsy_i\|^2  =\|\tsy_{i-1}+ \cB \tsz_i \|^2 &= \|\tsy_{i-1}\|^2+ \| \cB \tsz_i \|^2 + 2  \tsy_{i-1} \tran \cB \tsz_i \nonumber \\
	&\overset{\eqref{error_primal_ed}}{=} \|\tsy_{i-1}\|^2+ \| \tsz_i \|^2_{\cB^2} - 2    \|\cB \tsy_{i-1}\|^2 \nonumber \\ 
	& \quad +2   \tsy_{i-1}\tran \cB  \left((I-\cC)\tsw_{i-1}-\mu \big(\grad \cJ(\sw_{i-1})-\grad \cJ(\sw^\star) \big)\right) \label{er_sq_dual_ed}
}
Adding  equation \eqref{er_sq_dual_ed}  to \eqref{er_sq_primal_ed} and rearranging, we get 
\eq{
\|\tsz_i\|^2_{\cQ} \hspace{-0.6mm}+\hspace{-0.6mm} \|\tsy_i\|^2 \hspace{-0.6mm}=\hspace{-0.6mm} \|(I-\cC)\tsw_{i-1} \hspace{-0.6mm}- \hspace{-0.6mm} \mu \big(\grad \cJ(\sw_{i-1})\hspace{-0.6mm}-\hspace{-0.6mm}\grad \cJ(\sw^\star) \big)\|^2 \hspace{-0.6mm}+\hspace{-0.6mm} \|\tsy_{i-1}\|^2 \hspace{-0.6mm}-\hspace{-0.6mm}      \|\cB \tsy_{i-1}\|^2 \label{err_sum_ed}
}
where $\cQ = I -   \cB^2$ is positive semi-definite from \eqref{eq:asump_penalty}. Since $\sy_0 = 0$ and $\sy_i = \sy_{i-1} +  \cB \ssz_i$, we know $\sy_i\in \mbox{range}(\cB)$ for any $i$. Thus,  both $\sy_i$ and $\sy_b^\star$ lie in the range space of $\cB$, and it  holds that $
\|\cB \tsy_{i-1}\|^2 \geq 
\underline{\sigma}(\cB^2) \|\tsy_{i-1}\|^2 $. Therefore, we can bound \eqref{err_sum_ed} by
\eq{
	\|\tsz_i\|^2_{\cQ}+ \|\tsy_i\|^2 
	& \le\|\tsw_{i-1}-\mu \big(\grad \cJ(\sw_{i-1})-\grad \cJ(\sw^\star)+{1 \over \mu}\cC \tsw_{i-1}  \big)\|^2 \hspace{-0.5mm}+\hspace{-0.5mm} (1- \underline{\sigma}(\cB^2))\|\tsy_{i-1}\|^2 \label{err_sum1_ed}
}
Also, since $\cJ(\sw)+{1 \over 2 \mu}\|\sw\|^2_{\cC}$ is  $\delta_{\mu}=\delta+{1 \over \mu} \sigma_{\max}(\cC)$-smooth, it holds that \cite[Theorem 2.1.5]{nesterov2013introductory}:
\eq{
\| \big(\grad \cJ(\sw_{i-1})-\grad \cJ(\sw^\star)+{1 \over \mu}\cC \tsw_{i-1}  \big)\|^2 \leq  \delta_{\mu} \tsw_{i-1}\tran  \big(\grad \cJ(\sw_{i-1})-\grad \cJ(\sw^\star)+{1 \over \mu}\cC \tsw_{i-1}  \big)
}
Using this bound, it can be easily verified that:
\eq{
&\|\tsw_{i-1}-\mu \big(\grad \cJ(\sw_{i-1})-\grad \cJ(\sw^\star)+{1 \over \mu}\cC \tsw_{i-1}  \big)\|^2 \nnb
&\leq \|\tsw_{i-1}\|^2 - \mu (2-\mu \delta_{\mu} ) \tsw_{i-1}\tran  \big(\grad \cJ(\sw_{i-1})-\grad \cJ(\sw^\star)+{1 \over \mu}\cC \tsw_{i-1}  \big) \nnb
&\leq \big(1-  \mu \nu (2- \mu\delta_{\mu} )\big) \|\tsw_{i-1}\|^2=\big(1-  \mu \nu (2-\sigma_{\max}(\cC)-\mu \delta )\big) \|\tsw_{i-1}\|^2
}  
where in the last step  we used  the fact that $2-\mu\delta_\mu > 0$, which follows from the condition $\mu<(2-\sigma_{\max}(\cC))/\delta$, and the fact that  $\cJ(\sw)+{1 \over 2 \mu}\|\sw\|^2_{\cC}$ is $\nu$-strongly convex.   Thus, we can substitute the previous   inequality  in \eqref{err_sum1_ed} and get
\eq{
	\|\tsz_i\|^2_{\cQ} \hspace{-0.5mm}+\hspace{-0.5mm} \|\tsy_i\|^2 
	& \le \big(\hspace{-0.5mm} 1 \hspace{-0.5mm}-\hspace{-0.5mm} \mu \nu (2-\sigma_{\max}(\cC)-\mu \delta ) \hspace{-0.5mm}\big)\hspace{-0.5mm}\|\tsw_{i-1}\|^2 \hspace{-0.5mm}+ (1- \underline{\sigma}(\cB^2))\|\tsy_{i-1}\|^2 \label{err_sum1_ed-2_ed}
} 
From \eqref{error_prox_ed} and the nonexpansive property of the proximal operator, we have
\eq{
	\|\tsw_i\|^2 &= \|{\rm \bf prox}_{\mu \cR}\big(\bar{\cA} \ssz_i \big)-{\rm \bf prox}_{\mu \cR}(\bar{\cA} \ssz^\star) \|^2 \leq \|\bar{\cA} \tsz_i\|^2 \leq \| \tsz_i\|^2_{\cQ} \label{prox_bound_last}
}
where the last step holds because of condition \eqref{eq:asump_penalty} so that $\|\bar{\cA} \tsz_i\|^2=\| \tsz_i\|^2_{\bar{\cA}^2} \leq \| \tsz_i\|^2_{\cQ}$. Substituting \eqref{prox_bound_last} into \eqref{err_sum1_ed-2_ed}  we reach our result.  Finally we note that:
\eq{
\big(\hspace{-0.5mm} 1 \hspace{-0.5mm}-\hspace{-0.5mm} \mu \nu (2-\sigma_{\max}(\cC)-\mu \delta ) \hspace{-0.5mm}\big) < 1 \iff \mu < {2-\sigma_{\max}(\cC) \over \delta}
}
\end{proof}
\noindent  An interesting choice of $\bar{\cA}$, $\cB$, and $\cC$ is the class  with $\cC=0$. For $\cC=0$, which is the case for  exact diffusion \eqref{exact-diffusion} and Aug-DGM (ATC-DIGing) \eqref{atc_DGM},  the step size bound in Theorem \ref{theorem_lin_convergence} becomes $\mu<{2 \over \delta}$, which is independent of the network and as large as the centralized proximal gradient descent.   Moreover, for  $\cC=0$   the convergence rate becomes $\gamma= \max \{ 1 \hspace{-0.5mm}-\hspace{-0.5mm} \mu \nu (2-\mu \delta ) ,1 - \underline{\sigma}(\cB^2)\}<1$, which separates the network effect  from the cost function. If we further choose $\bar{\cA}=\cA^j$ and and $\cB^2=I-\cA^j$ for integer $j\geq 1$, then we have $1 - \underline{\sigma}(\cB^2)=\lambda_2(\cA^j) \rightarrow 0$ as $j \rightarrow \infty$ where $\lambda_2(\cA^j)$ is the second largest eigenvalue of $\cA^j$ . Thus, the convergence rate $\gamma=   1 \hspace{-0.5mm}-\hspace{-0.5mm} \mu \nu (2-\mu \delta )$ can match the rate of centralized algorithms for large $j$.  A similar conclusion appears for NIDS \cite{li2017nids} but for the smooth case, which is subsumed in our framework. 
\begin{remark}[\sc Network Effect]{\rm
\label{remark:numberofag}
 The convergence rate depends on the network graph through the terms $\underline{\sigma}(\cB^2)$ and $\sigma_{\max}(\cC)$. Given a certain graph, it will depend on the number of agents {\em indirectly} as we now explain.  If we choose $\cB^2=I-\cA$ and $\cC=0$ where $\cA$ is constructed as in Section \ref{sec:combina:matrix} and satisfy Assumption \ref{assump_combination}. Then, we have that $1 - \underline{\sigma}(\cB^2)=\lambda_2(\cA)$  where $\lambda_2(\cA)$ denotes the second largest eigenvalue of $\cA$.    For a cyclic network it holds that $\lambda_2(\cA)=1-\cO(1/K^2)$. For a grid network we have $\lambda_2(\cA)=1-\cO(1/K)$. For a fully connected network, we can choose $\cA=  {1 \over K} \one \one\tran$ so that $\lambda_2(\cA)=0$.  In this case, we can also choose $\bar{\cA}=  {1 \over K} \one \one\tran$ and the primal updates in \eqref{alg_prox_ATC_framework} becomes so that each agent updates its vector via a proximal gradient descent update on the objective function given in problem \eqref{decentralized1}.      \qd
}
\end{remark}
\section{Simulations on real data} \label{sec-simulation}  
In this section we test the performance of three different instances of the proposed method \eqref{alg_prox_ATC_framework} against some state-of-the-art algorithms. We consider the following sparse logistic regression problem:
\eq{ 
	\min_{w\in \real^M} \frac{1}{K}\sum_{k=1}^K J_k(w) + \rho \|w\|_1	\quad \mbox{where}\quad J_k(w) = \frac{1}{L}\sum_{\ell=1}^{L}\ln(1+\exp(-y_{k,\ell} x_{k,\ell}\tran w)) + \frac{\lambda}{2}\|w\|^2 \nonumber
}
where $\{x_{k,\ell}, y_{k,\ell}\}_{\ell=1}^L$ are local data kept by agent $k$ and $L$ is the size of the local dataset. We consider three real datasets: Covtype.binary, MNIST, and CIFAR10. The last two datasets have been transformed into binary classification problems by considering data with two labels, digits two and four (`2' and `4') classes for MNIST, and cat and dog classes for CIFAR-10. In Covtype.binary we use 50,000 samples as training data and each data has dimension 54. In MNIST we use 10,000 samples as training data and each data has dimension 784. In CIFAR-10 we use 10,000 training data and each data has dimension 3072.  All features have been preprocessed and normalized to the unit vector with sklearn's normalizer\footnote{\url{https://scikit-learn.org}}.
\begin{figure*}[h!]
	\centering
	\includegraphics[scale=0.35]{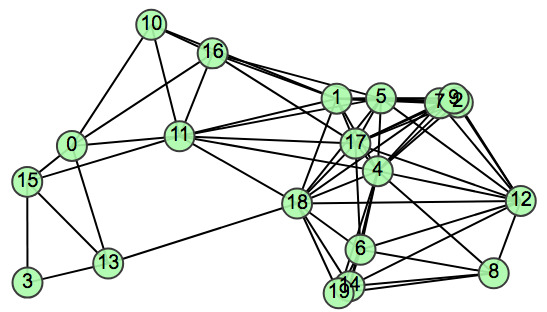}
	\caption{The network topology used in the simulation.}
	\label{fig-network}
\end{figure*}

For the network, we generated a randomly connected network with $K=20$ agents, which is shown in Fig. \ref{fig-network}. The associated combination matrix $A$ is generated according to the Metropolis rule \cite{sayed2014nowbook}.  For all simulations, we assign data evenly to each  agent. We set $\lambda=10^{-4}$ and $\rho=2\times10^{-3}$ for Covtype, $\lambda=10^{-2}$ and $\rho=5\times10^{-4}$ for CIFAR-10, and $\lambda=10^{-4}$ and $\rho=2\times10^{-3}$ for MNIST. The simulation results are shown in Figure \ref{fig-lr}. The decentralized implementations of Prox-ED, Prox-ATC I, and prox-ATC II are given in Appendix \ref{supp_equiva_represent_prox}.    For each algorithm, we tune the step-sizes manually to achieve the best possible convergence rate.   We notice that the performance of each algorithm  differs in each data set and Prox-ED performs the best in our simulation setup. The $x$-axis in these plots is in terms of rounds of communication per iteration. Note that Prox-ATC I and Prox-ATC II require two rounds of communication per iteration compared to only one round for all other algorithms -- see Remark \ref{remak:sharing-variable}.
\begin{figure*}[t!]
	\centering
	\includegraphics[scale=0.35]{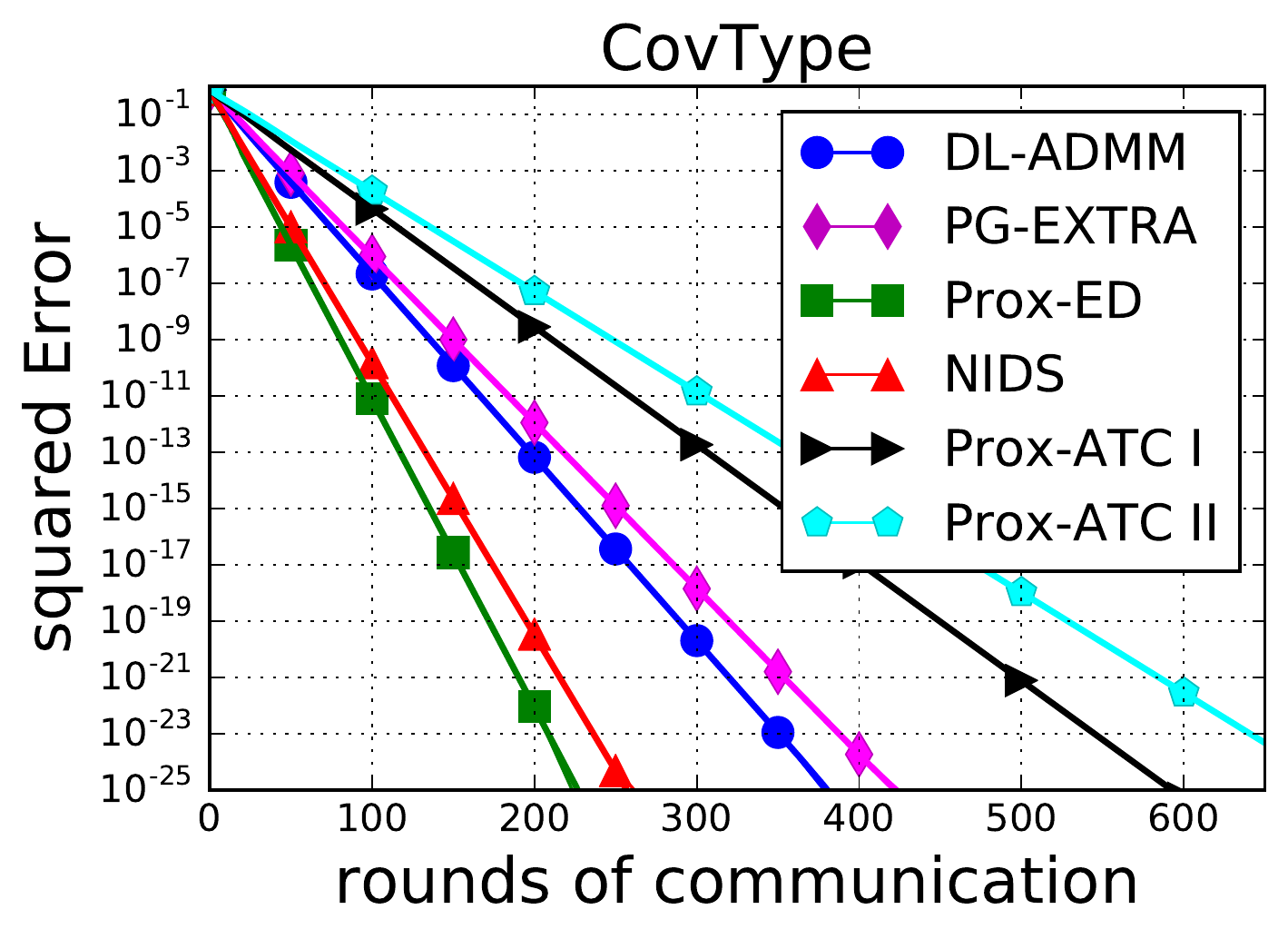}
	\includegraphics[scale=0.35]{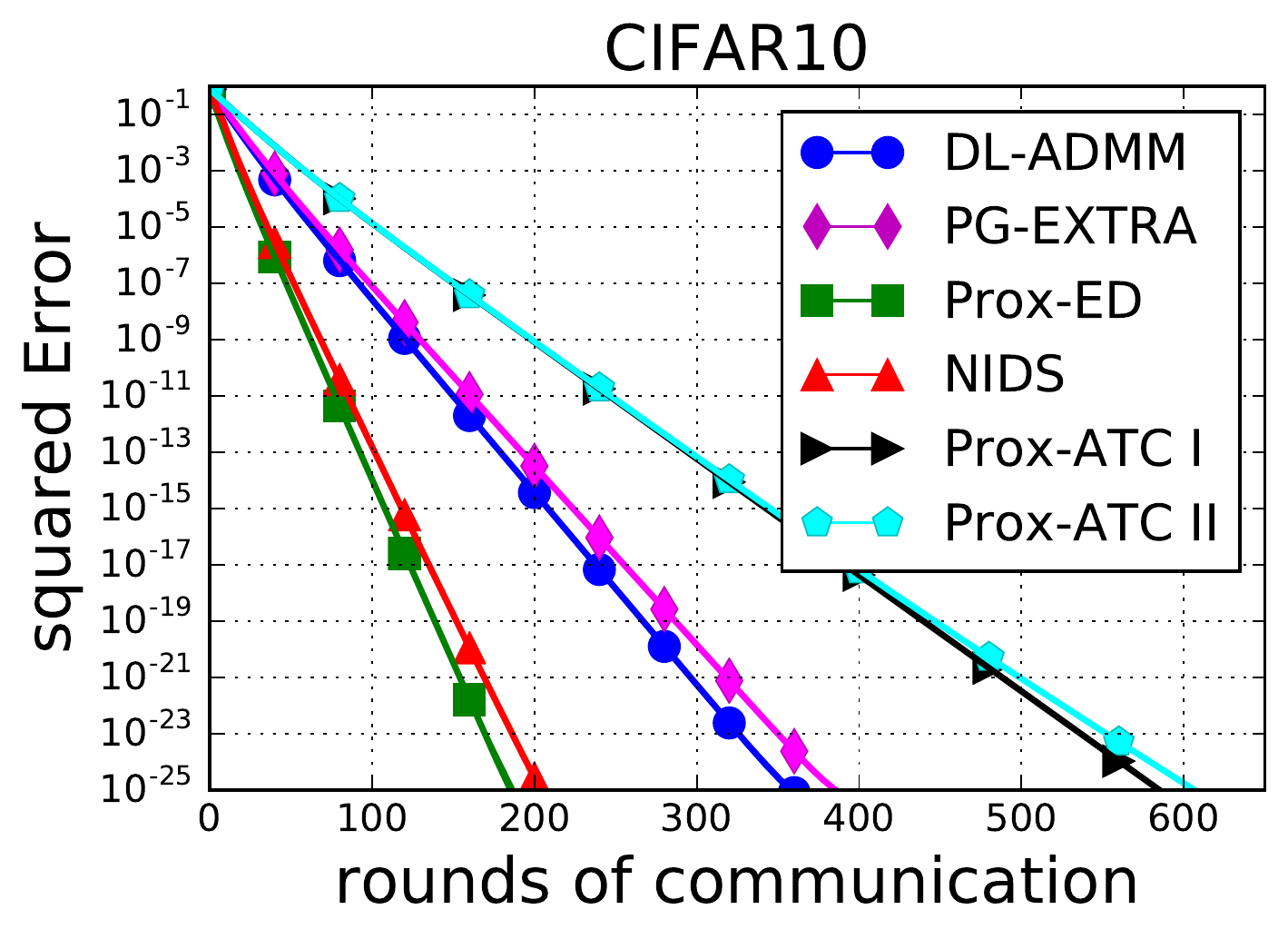}
	\includegraphics[scale=0.35]{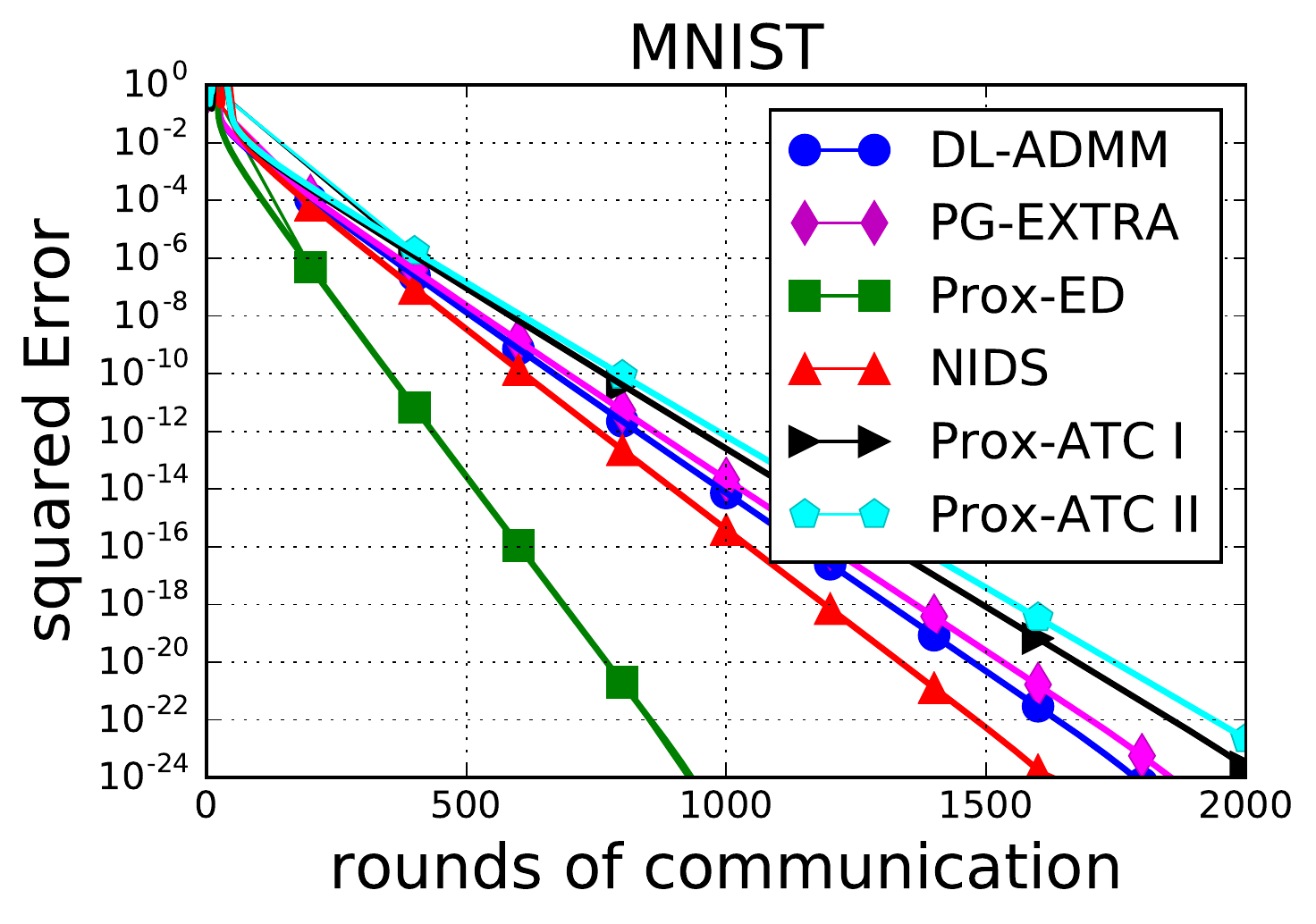}
	\caption{ \footnotesize Simulation results. The $y$-axis indicates the relative squared error $\sum_{k=1}^{K}\|w_{k,i} - w^\star\|^2/\|w^\star\|^2$. Prox-ED refers to \eqref{alg_prox_ATC_framework} with $\bar{\cA}=0.5 (I+\cA)$, $\cB^2=0.5 (I- \cA)$, and $\cC=0$. Prox-ATC I refers to \eqref{alg_prox_ATC_framework} with $\bar{\cA}=\cA^2$, $\cB=I-\cA$, and $\cC=0$. Prox-ATC II refers to \eqref{alg_prox_ATC_framework} with $\bar{\cA}=\cA$, $\cB=I-\cA$, and $\cC=I-\cA$. DL-ADMM  \cite{chang2015multi}, PG-EXTRA \cite{shi2015proximal}, NIDS \cite{li2017nids}.
}
	\label{fig-lr}
\end{figure*} 

\section{Separate non-smooth terms: sublinear rate} \label{sec:sublinearbound}
 In this section, we will show that if each agent owns a different local non-smooth term, then {\em exact} global linear convergence cannot be attained in the worst case (for all problem
instances) although it can still be possible for some special cases.  Consider the more general problem with agent specific regularizers:
\eq{
\label{eq:separarate-regularizer}
\min_{w\in \real^M}\ \frac{1}{K}\sum_{k=1}^{K}J_k(w)+R_k(w),
}
where $J_k(w)$ is a strongly convex smooth function and $R_k(w)$ is non-smooth convex with closed form proximal mappings (each $J_k(w)$ and $R_k(w)$ are further assumed to be closed and proper functions). Although many algorithms (centralized and decentralized) exist that solve \eqref{eq:separarate-regularizer},  none have been shown to achieve linear convergence in the presence of general non-smooth proximal terms  $R_k(w)$.  In the following, by tailoring the results from \cite{woodworth2016tight}, we show that this is not possible when  having access to the proximal mapping of each individual non-smooth term   $R_k(w)$ separately.
\subsection{Sublinear Lower Bound} \label{sec-sublinear-bound}
Let $\mathcal{H}$ be a deterministic algorithm that queries
\[
\{J_k(\cdot), R_k(\cdot), \nabla J_k(\cdot), {\rm \bf prox}_{\mu_{i,k}   R_k}(\cdot)
\,|\,
\mu_{i,k}>0,\,
k=1,\dots,K
\}
\]
once for each iteration $i=0,1,\dots$.
To clarify, the scalar parameter $\mu_{i,k}>0$ can differ for $i=0,1,\dots$ and $k=1,\dots,K$ or they can be constants (e.g.\ $\mu_{i,k}=\mu >0$).
Note that $\mathcal{H}$ has the option to combine the queried values in any possible combination (e.g., it can only use certain information from certain communications). Thus, $\mathcal{H}$ includes decentralized algorithms in which communication is restricted to edges on a graph.

Consider the specific instance of \eqref{eq:separarate-regularizer}
\eq{
\min_{w\in \real^M} \ F_\nu(w)=\frac{\nu}{2}\|w\|^2+\frac{1}{K}\sum_{k=1}^{K} R_k(w)
\label{cost_F_nu}}
where $\nu>0$ and $J_k(w)= \frac{\nu}{2K}\|w\|^2$. Assume $R_k(w)<\infty$ if and only if $\|w\|\le B$ and $|R_k(w_1)-R_k(w_2)|\le G\|w_1-w_2\|$ for all $w_1,w_2$ (where $B$ and $G$ are some positive constants) such that 
$\|w_1\|\le B$ and $\|w_2\|\le B$. To prove that linear convergence is not possible, we will reduce our setup to $\min_{w\in \real^M}\, F_0(w)$, which has a known lower bound \cite{woodworth2016tight}.
Let $\mathcal{H}_o$ be a deterministic algorithm that queries
\[
\{ R_k(\cdot), {\rm \bf prox}_{\mu_{i,k} R_k(\cdot)}(\cdot)
\,|\,
\mu_{i,k}>0,\,
k=1,\dots,K
\}
\]
once for each iteration $i=0,1,\dots$ and communicates through a fully connected network.
The following result is a special case of the more general result  \cite[Theorem~1]{woodworth2016tight}.
\begin{theorem} 
\label{thm:woodworth_lower_bnd}
Let $0<B$, $0<G$, $2 \leq K$, and $0<\varepsilon<GB/12$.
For a large enough problem dimension $M=\mathcal{O}(KGB/\varepsilon)$, 
the algorithm $\mathcal{H}_o$
(in the worst case) requires $\mathcal{O}(GB/\varepsilon)$ or more iterations to find a $\hat{w}$ such that
$F_0(\hat{w})-\inf_{w}F_0(w)<\varepsilon$.
\end{theorem}

\noindent We argue that algorithm $\mathcal{H}$ cannot be too efficient at solving $\min_{w}\, F_\nu(w)$ with $\nu>0$ as otherwise it can be used  to efficiently solve $\min_{w}\, F_0(w)$ and contradict Theorem~\ref{thm:woodworth_lower_bnd}.

\begin{theorem}
\label{thm:main_lower_bound}
Let $0<\nu$, $0<B$,  $0<G$, $2 \leq K$, and $0<\varepsilon<G^2/(288\nu)$.
For a large enough problem dimension $M=\mathcal{O}(KG/\sqrt{\nu\varepsilon})$,
the algorithm $\mathcal{H}$ (in the worst case) requires 
$\mathcal{O}(G/\sqrt{\nu\varepsilon})$ or more iterations to find a $\hat{w}$ such that
$F_\nu(\hat{w})-\inf_{w}F_\nu(w)<\varepsilon$.

\end{theorem}

\begin{proof}
 This argument modifies the proof of  \cite[Theorem~2]{woodworth2016tight},  which makes  a similar but slightly different claim. Let $\nu=\varepsilon/B^2$ and  $w^\star_\nu$ denotes the minimizer of $F_\nu$. 
Assume for contradiction that $\mathcal{H}$ can find a $\hat{w}$ such that
\eq{
F_\nu(\hat{w})-F_\nu(w^\star_\nu) < \frac{\varepsilon}{2} \label{eq:contrad}
}
in $o(G/\sqrt{\nu\varepsilon})$ iterations.
Note that for all $w$ such that $\|w\|\le B$, it holds from \eqref{cost_F_nu}:
\eq{
F_\nu(w)\le F_0(w)+\frac{\nu B^2}{2}= F_0(w)+\frac{\varepsilon}{2}. \label{eq:FvleqFo}
}
Putting these together, we get
\eq{
F_0(\hat{w})-F_0(w^\star_0)-\frac{\varepsilon}{2} \overset{\eqref{eq:FvleqFo}}{\leq} F_0(\hat{w})-F_\nu(w^\star_0) \overset{(a)}{\leq} F_\nu(\hat{w})-F_\nu(w^\star_\nu)  < \frac{\varepsilon}{2},
}
where in step (a) we used $F_0(w)\le F_\nu(w)$ and $F_\nu(w^\star_\nu) \leq F_\nu(w^\star_0) $. We conclude that $F_0(\hat{w})-F_0(w^\star_0) < \varepsilon$. 
Since $\nabla J_k(\cdot)={\nu \over K} I$ is just a scaled identity, querying $\nabla J_k(\cdot)$ does not provide a new direction that $\mathcal{H}_o$ could otherwise not use. Thus, algorithm $\mathcal{H}$ applied to minimizing $F_\nu$ is an instance of algorithm $\mathcal{H}_o$. This means that we have an algorithm for minimizing $F_0$ in 
$o(G/\sqrt{\nu\varepsilon})=o(GB/\varepsilon)$ iterations, which contradicts Theorem~\ref{thm:woodworth_lower_bnd}. Note that  $0 < \varepsilon <GB/12$ from Theorem~\ref{thm:woodworth_lower_bnd} and by using $\nu=\varepsilon/B^2$ we require $0< \varepsilon <G^2/(288\nu)$ (an extra factor of $2$ appears because of \eqref{eq:contrad}). 
\end{proof}

\begin{corollary}
For the problem setup of \eqref{eq:separarate-regularizer} with strongly convex $J_k(\cdot)$ for all $k=1,2,\dots,K$,
any algorithm that accesses the functions through
evaluations of $J_k(\cdot)$ and $R_k(\cdot)$, the gradients of $J_k(\cdot)$, and proximal operators of $R_k(\cdot)$
is not globally linearly convergent (in the worst case).
\end{corollary}
\begin{remark} \label{remark:lowerbound_dimension} 
{\rm The lower bound of Theorem \ref{thm:main_lower_bound} is \emph{dimension independent} in the same way other Nesterov-type lower bounds are \cite{woodworth2016tight,nesterov2013introductory}.
The result implies that it is not possible to establish linear convergence of $\cH$ with a rate depending on $\nu$ and $G$, but not on the problem dimension $K$.
That said, a dimension dependent linear convergence may be established.
For example,  {\em eventual} linear convergence\footnote{A sequence $\{x_i\}_{i=0}^\infty$ has eventual linear convergence  to $x^\star$  if there exists a sufficiently large $i_o$ such that $\|x_i-x^\star\| \leq \gamma^i C$ for some $C>0$ and all $i \geq i_o$.} has been established in \cite{latafat2017new} when  the functions $\{J_k(\cdot),R_k(\cdot)\}$ are piecewise linear quadratic.
This result does not contradict our result as the linear rate and the number of iterations needed to observe the linear rate are \emph{dependent} on the problem dimension.
Our linear convergence result of Theorem \ref{theorem_lin_convergence} is dimension independent as it holds for any dimension $M$.
} \qd
\end{remark}
\subsection{Numerical Counterexample}
In this section, we numerically show that linear convergence to the exact solution $w^\star$ is not possible in general.   We consider an instance of \eqref{eq:separarate-regularizer} with $K = 2$,  $M$  is a very large  even number, and quadratic smooth terms $J_k(w)=\eta/2 \|w\|^2$ for some $\eta >0$. We let the non-smooth terms be
\begin{subequations}\label{Rk}
	\eq{
		R_1(w)&=|\sqrt{2}w(1)-1| \hspace{-0.5mm}+\hspace{-0.5mm} |w(2)-w(3)| \hspace{-0.5mm}+\hspace{-0.5mm} |w(4)-w(5)| \hspace{-0.5mm}+\hspace{-0.5mm} \cdots \hspace{-0.5mm}+\hspace{-0.5mm}|w(M\hspace{-0.5mm}-\hspace{-0.5mm}2)\hspace{-0.5mm}-\hspace{-0.5mm}w(M\hspace{-0.5mm}-\hspace{-0.5mm}1)| \\
		R_2(w)&=|w(1)-w(2)|+|w(3)-w(4)|+\cdots
		+|w(M-1)-w(M)| 
	}
\end{subequations}
 Both ${\bf prox}_{R_1}$ and ${\bf prox}_{R_2}$ have closed forms --- see Appendix \ref{app_counter_example_proximal} for details.
 The above construction is related to the one in \cite{arjevani2015communication}, which was used to derive lower bounds for a different class of algorithms as explained in the introduction.

In the numerical experiment, we test the performance of two well known decentralized proximal methods, PG-EXTRA \cite{shi2015proximal} and DL-ADMM \cite{chang2015multi,aybat2018distributed}. Note that the structure of updates \eqref{alg_prox_ATC_framework}  are designed to handle a common non-smooth term case only, which is why we do not test it in this numerical counterexample. We set $M=2000$ and $\eta = 1$. The step-sizes for both PG-EXTRA and DL-ADMM are set to $0.005$. The combination matrix is set as $A = \frac{1}{2}\mathds{1}_2 \mathds{1}_2\tran$.  The numerical results in the left plot of Fig. \ref{fig-counter-example} shows that both PG-EXTRA and DL-ADMM converge sublinearly to the solution. In particular, we see that the error curves after around $10^3$ iterations has sublinear convergence. The right plot in Fig. \ref{fig-counter-example} shows the squared error where both $x$-axis and $y$-axis are in logarithmic scales. In this scale, a straight line indicates a sublinear rate, which is clearly visible after around $10^3$ iterations.  
No global linear convergence is observed in the simulation for sufficiently large dimension $M$ and algorithms independent of $M$, which is consistent with our discussion in Remark \ref{remark:lowerbound_dimension}. 
\begin{figure*}[h!]
\centering
	\includegraphics[scale=0.55]{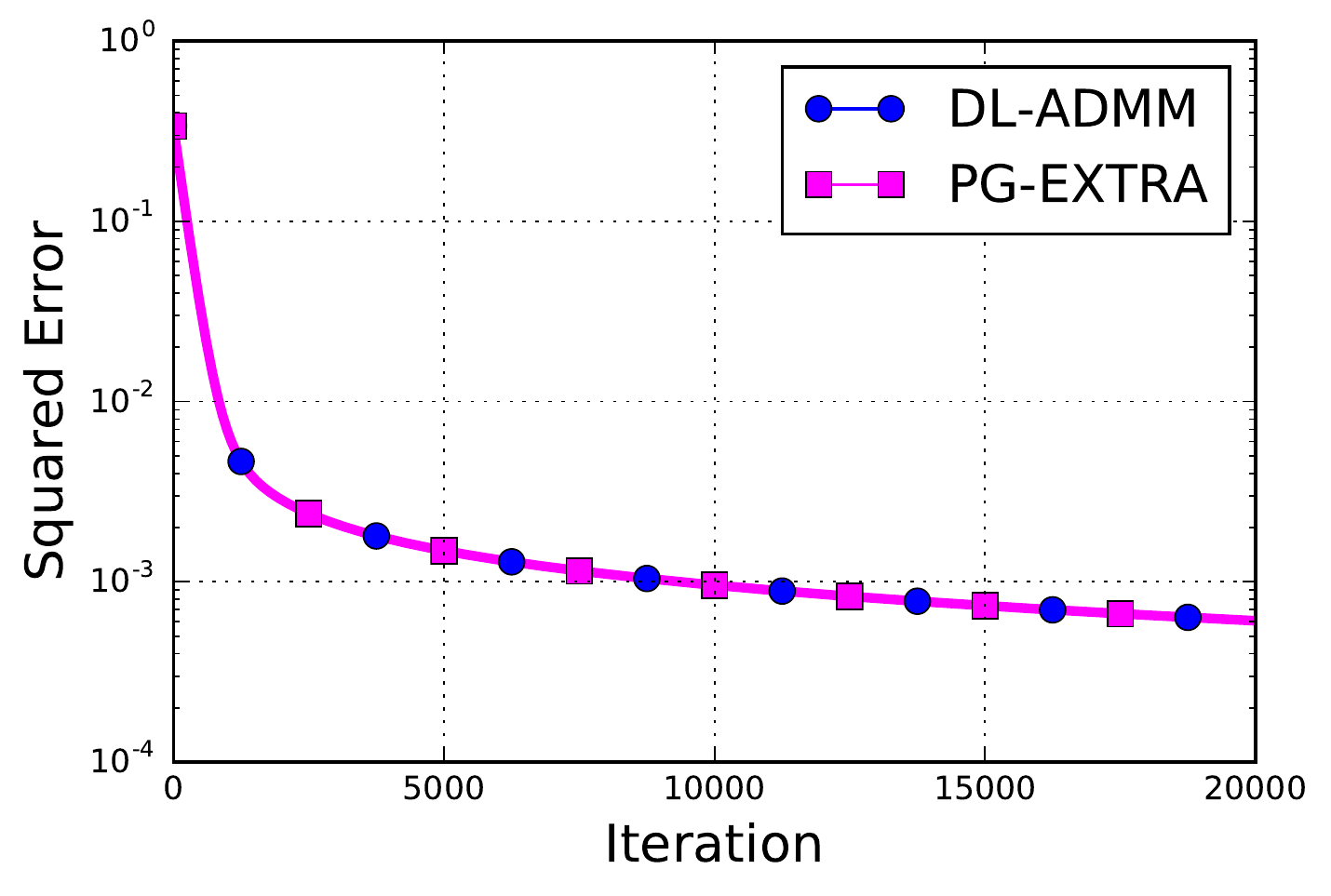}
	\includegraphics[scale=0.55]{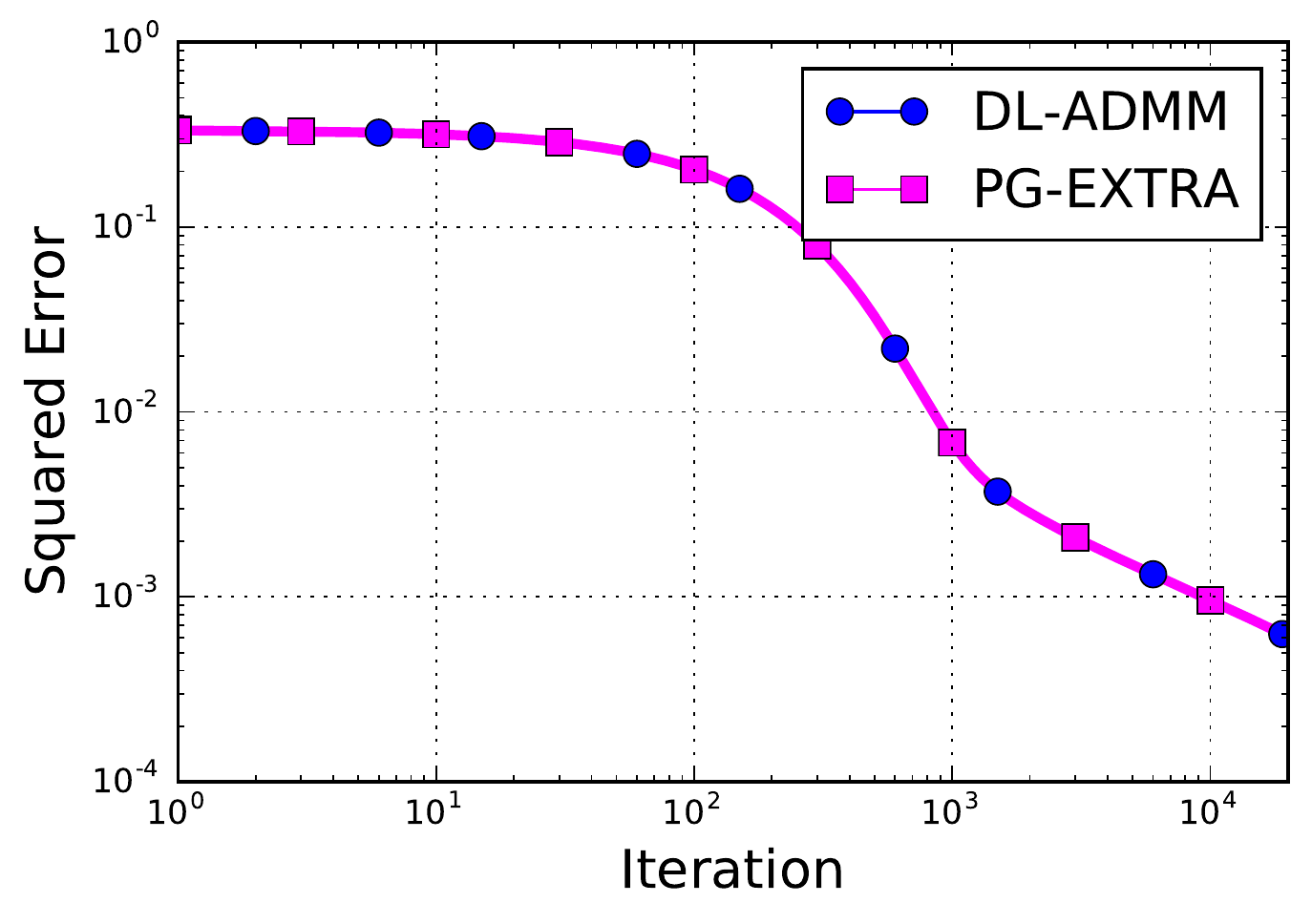}
	\caption{ \footnotesize Numerical counterexample simulations. Both $y$-axis and $x$-axis are in logarithmic scales in the right plot.  PG-EXTRA \cite{shi2015proximal} and DL-ADMM \cite{chang2015multi,aybat2018distributed} converge sublinearly to the solution
 of the proposed numerical counterexample.}
	\label{fig-counter-example}
\end{figure*} 
\section{Concluding Remarks}
In this work, we proposed a proximal primal-dual algorithmic framework, which subsumes many existing algorithms in the smooth case, and established its linear convergence under strongly-convex objectives.     Our analysis provides wider step-size conditions than many existing works, which provides insightful indications on the performance of each algorithm. That said, these step-size bound comes  at the expense of stronger assumption on the combination matrices -- see Remark \ref{remark:conv_conditions}.  It is therefore of interest to study the interrelation between the step-sizes and combination matrices for linear convergence.  Regarding the discussion below Theorem \ref{theorem_lin_convergence}, a useful future direction is to study how to optimally choose $\bar{\cA}$, $\cB$, and $\cC$ as a function of $\cA$ to get the best possible convergence rate while balancing the communication cost per iteration.

\medskip
{\small
\bibliographystyle{ieeetr}
\bibliography{myref_atc} 
}
\newpage
\appendices
\section{Proof of Lemma \ref{lemma:existence_fixed_optimality}}\label{supp_lemma_fixed}
To show the existence we will construct a point $(\sw^\star, \sy^\star, \sz^\star)$ that satisfies equations \eqref{p-d_ed-star}--\eqref{prox_step_ed-star}. Since each $J_k(w)$ is strongly convex, there exists a unique solution $w^\star$ for problem \eqref{decentralized1}, i.e., 
$
0 \in \frac{1}{K}\sum_{k=1}^K \grad J_k(w^\star) + \partial R(w^\star).	
$
This also indicates that there must exist a subgradient $r^\star \in \partial R(w^\star)$ such that 
\eq{
\frac{1}{K}\sum_{k=1}^K \grad J_k(w^\star) + r^\star = 0}
 Now we define $z^\star \define (\mu/K)r^\star + w^\star$, it holds that $r^\star/K + (w^\star - z^\star)/\mu = 0$, i.e., $0\in (1/K)\partial R(w^\star) + (1/\mu) (w^\star - z^\star)$. This implies that 
\eq{
w^\star = \argmin_w \left\{ \frac{1}{K} R(w) + \frac{1}{2\mu}\|w - z^\star\|^2 \right\}. \label{prox_wstar}	
}
We next define $\sw^\star \define \one_{K} \otimes w^\star$ and $\sz^\star \define \one_{K} \otimes z^\star $. Since $\sz^\star = \mathds{1}_K \otimes z^\star$, it belongs to the null space of $\cB$ so that  $\cB\sz^\star = 0$ and, moreover, $\bar{\cA} \sz^\star=\sz^\star$  since $\bar{\cA}=\bar{A} \otimes I_M$ where $\bar{A}$ is doubly stochastic. Therefore,  relation \eqref{prox_wstar} implies that  equation \eqref{prox_step_ed-star} holds. It remains to construct $\sy^\star$ that satisfies equation \eqref{p-d_ed-star}. Note that
\eq{
(\mathds{1}_N\otimes I_M)\tran \big(\sw^\star - \sz^\star - \mu \grad \cJ(\sw^\star)\big) = - \mu r^\star - \frac{\mu}{K} \sum_{k=1}^K \grad  J_k(w^\star) = 0, \label{2387sdhsdh}
}
where the last equality holds because $w^\star$ is the optimal solution of problem \eqref{decentralized1}. Equation \eqref{2387sdhsdh} implies  
\eq{
\frac{1}{\mu}\big( \sw^\star - \sz^\star - \mu \grad \cJ(\sw^\star) \big) \in \mbox{Null}(\mathds{1}_N\otimes I_M) =\mbox{Null}(\cB)^{\perp}= \mbox{Range}(\cB).
}
where $\perp$ denotes the orthogonal complement. Therefore, there exist a vector $\sy^\star$ satisfying equation \eqref{p-d_ed-star}.

We now show  that any fixed point  is of the form $\sw^\star = \one_{K} \otimes w^\star$ and $w^\star$ is the solution to problem \eqref{decentralized1}. From \eqref{d-a_ed-star} and \eqref{consensus-condition-both}, it holds that the block elements of $\sz^\star$ are equal to each other, i.e. $z_1^\star = \cdots = z_K^\star$, and we denote each block element by $z^\star$.  Thus,   $\bar{\cA} \ssz^\star=\ssz^\star= \one_K \otimes z^\star$ because $\bar{\cA}=\bar{A} \otimes I_M$ where $\bar{A}$ is doubly stochastic.
 Therefore, from  \eqref{prox_step_ed-star} and the definition of the proximal operator it holds that
 \eq{
 w_k^\star =\argmin_{w_k} \ \{R(w_k)/K + \|w_k - z^\star\|^2/2\mu\}
\label{238impl} }
  where we used $z_k^\star = z^\star$ for each $k$. Thus, we must have $w_1^\star = \cdots = w_K^\star \define w^\star$. It is easy to verify that \eqref{238impl} implies
\eq{
0 \in \partial R(w^\star)/K + (w^\star - z^\star)/\mu. \label{238sdhsd}
} 
 Multiplying $(\mathds{1}_K\otimes I_M)\tran$ from the left to both sides of equation \eqref{p-d_ed-star}, we get 
\begin{align}
	K z^\star = K w^\star - \frac{\mu}{K}\sum_{k=1}^K \grad J_k(w^\star) \label{cxbwebdshbds}
\end{align}
Combining \eqref{238sdhsd} and \eqref{cxbwebdshbds}, we get
$0 \in 	\frac{1}{K}\sum_{k=1}^K \grad J_k(w^\star) + \partial R(w^\star)$. 
Thus, $w^\star$ is the unique solution to problem \eqref{decentralized1}. Due to the uniqueness of $w^\star$, we see from \eqref{cxbwebdshbds} that $z^\star$ is unique. Consequently, $\sw^\star=\one_K \otimes w^\star $ and $\sz^\star=\one_K \otimes z^\star$ must be unique.
\section{Equivalent Representation}
\label{supp_equiva_representation}
\subsection{Aug-DGM (ATC-DIGing)}
Here we show that \eqref{atc_DGM} is equivalent to \eqref{atc_DGM_eliminate}. From \eqref{atc-dgm1} we have
\eq{
\sw_i-\cA \sw_{i-1} &=\cA \bigg(\sw_{i-1}-\cA \sw_{i-2}-\mu \big(  \ssx_{i-1}-\cA \ssx_{i-2} \big) \bigg) \nnb
& \overset{\eqref{atc-dgm2}}{=} \cA \bigg(\sw_{i-1}-\cA \sw_{i-2}-\mu \cA \big(\grad \cJ(\sw_{i-1})-\grad \cJ(\sw_{i-2})\big)  \bigg) \nonumber
}
Rearranging the previous equation we get:
\eq{
\sw_i = \cA \bigg(2\sw_{i-1}-\cA \sw_{i-2}-\mu \cA \big(\grad \cJ(\sw_{i-1})-\grad \cJ(\sw_{i-2})\big)  \bigg) \nonumber
}
which is recursion \eqref{atc_DGM_eliminate}.
 
\subsection{ATC-Tracking}
In a similar manner we can show that \eqref{next} is equivalent to \eqref{next_eliminate}. From \eqref{next1} we have
\eq{
\sw_i-\cA \sw_{i-1} &=\cA \bigg(\sw_{i-1}-\cA \sw_{i-2}-\mu \big(  \ssx_{i-1}-\cA \ssx_{i-2} \big) \bigg) \nnb
& \overset{\eqref{next2}}{=} \cA \bigg(\sw_{i-1}-\cA \sw_{i-2}-\mu  \big(\grad \cJ(\sw_{i-1})-\grad \cJ(\sw_{i-2})\big)  \bigg) \nonumber
}
Rearranging the previous equation we get:
\eq{
\sw_i = \cA \bigg(2\sw_{i-1}-\cA \sw_{i-2}-\mu  \big(\grad \cJ(\sw_{i-1})-\grad \cJ(\sw_{i-2})\big)  \bigg) \nonumber
}
which is recursion \eqref{next_eliminate}.
\section{Implementation of \eqref{alg_prox_ATC_framework}}
\label{supp_equiva_represent_prox}
\subsection{ Prox-ED: $\bar{\cA}=0.5 (I+\cA)$, $\cB^2=0.5 (I- \cA)$, and $\cC=0$}
 Recursion \eqref{alg_prox_ATC_framework} with $\bar{\cA}=0.5 (I+\cA)$, $\cB^2=0.5 (I- \cA)$, and $\cC=0$ is equivalent to  the proximal exact diffusion (Prox-ED) recursion listed in \eqref{prox_diff_agent_adapt}--\eqref{prox_diff_agent_proximal}.	To see this, note that for $i=0$, it is straight forward to check that each block in \eqref{primal_prox_ATC_DIG} is the same as $w_{k,0}$ in \eqref{prox_diff_agent}. Now we will show the equivalence   for $i \geq 1$.  From \eqref{z_prox_ATC_DIG}, we know that:
 \eq{
 \ssz_i-\ssz_{i-1} &= \sw_{i-1}-\sw_{i-2}-\mu \big(\grad \cJ(\sw_{i-1})-\grad \cJ(\sw_{i-2})\big) -  \cB (\sy_{i-1}-\sy_{i-2}) \nnb
 &= \sw_{i-1}-\sw_{i-2}-\mu \big(\grad \cJ(\sw_{i-1})-\grad \cJ(\sw_{i-2})\big) -  \cB^2 \ssz_{i-1}
 }
where we used \eqref{dual_prox_ATC_DIG} in the last step.  Re-arranging and noting that $\cB^2=0.5 (I- \cA)$ we get
  \eq{
 \ssz_i&=\bar{\cA} \ssz_{i-1}+ \sw_{i-1}-\sw_{i-2}-\mu \big(\grad \cJ(\sw_{i-1})-\grad \cJ(\sw_{i-2})\big) 
 }
 By multiplying $\bar{\cA}$ to both sides of the previous equation and introducing $\ssx_i \define \bar{\cA} \ssz_i$ we get
  \eq{
 \ssx_i&=\bar{\cA} \bigg( \ssx_{i-1}+ \sw_{i-1}-\sw_{i-2}-\mu \big(\grad \cJ(\sw_{i-1})-\grad \cJ(\sw_{i-2})\big) \bigg)
 }
 Thus from \eqref{primal_prox_ATC_DIG} we get
 \begin{subequations} \label{prox_ex_diffusion_network}
 \eq{
  \ssx_i&= \bar{\cA} \bigg( \ssx_{i-1}+ \sw_{i-1}-\sw_{i-2}-\mu \big(\grad \cJ(\sw_{i-1})-\grad \cJ(\sw_{i-2})\big) \bigg)  \\
\sw_i&={\rm \bf prox}_{\mu R}( \ssx_i)  
} 
\end{subequations}
The above recursion is equivalent to \eqref{prox_diff_agent_adapt}--\eqref{prox_diff_agent_proximal}. This can be easily seen by substituting \eqref{prox_diff_agent_adapt}--\eqref{prox_diff_agent_correct} into \eqref{prox_diff_agent_combine}. 
\begin{algorithm}[h] 
\caption*{\textrm{\bf{Algorithm}} (Prox-ED)}
{\bf Setting}: Let $\bar{A}=[\bar{a}_{sk}]=(I_{K}+A)/2$. Initialize $x_{k,-1}=\psi_{k,-1}$ and $w_{k,-1}$ arbitrary. For every agent $k$, repeat for $i=0,1,2,...$
\begin{subequations}
\label{prox_diff_agent}
\eq{
\psi_{k,i}&=w_{k,i-1}-\mu \grad J_k(w_{k,i-1}) \label{prox_diff_agent_adapt} \\
z_{k,i}&= x_{k,i-1}+\psi_{k,i}-\psi_{k,i-1} \label{prox_diff_agent_correct} \\
 x_{k,i}&= \sum_{s \in \cN_k} \bar{a}_{sk} z_{s,i}   \label{prox_diff_agent_combine} \quad \textbf{(Communication step)} \\
w_{k,i} &= {\bf prox}_{\mu R}(x_{k,i}) \label{prox_diff_agent_proximal}
}
\end{subequations}
\end{algorithm}
 \subsection{Prox-ATC I: $\bar{\cA}= \cA^2$, $\cB^2= (I- \cA)^2$, and $\cC=0$}
 For the choice $\bar{\cA}= \cA^2$, $\cB^2= (I- \cA)^2$, and $\cC=0$, we can represent \eqref{alg_prox_ATC_framework} as listed in \eqref{prox_ATCI}. This can be seen by following the same approach as the previous subsection. To see this, note that with $\cB^2= (I- \cA)^2$ to get
  \eq{
 \ssz_i&=(2\cA-\cA^2) \ssz_{i-1}+ \sw_{i-1}-\sw_{i-2}-\mu \big(\grad \cJ(\sw_{i-1})-\grad \cJ(\sw_{i-2})\big) 
 }
 By multiplying $\cA^2$ to both sides of the previous equation and introducing $\ssx_i \define \cA^2 \ssz_i$ we get
  \eq{
 \ssx_i&=\cA \bigg( (2I-\cA) \ssx_{i-1}+ \cA \sw_{i-1}-\cA \sw_{i-2}-\mu \cA \big(\grad \cJ(\sw_{i-1})-\grad \cJ(\sw_{i-2})\big) \bigg)
 }
 Thus from \eqref{primal_prox_ATC_DIG} we have $
\sw_i={\rm \bf prox}_{\mu R}( \ssx_i)$. 
\begin{algorithm}[h] 
\caption*{\textrm{\bf{Algorithm}}: Prox-ATC I}
{\bf Setting}:  Initialize $x_{k,-1}=\psi_{k,-1}=0$ and $w_{k,-1}$ arbitrary. For every agent $k$, repeat for $i=0,1,2,...$
\begin{subequations} \label{prox_ATCI}
\eq{
\psi_{k,i}&= w_{k,i-1}-\mu \grad J_k(w_{k,i-1})  \\
z_{k,i}&= 2x_{k,i-1}- \sum_{s \in \cN_k}a_{sk}(  x_{s,i-1}-\psi_{s,i}+\psi_{s,i-1}) \quad &\textbf{(Communication step)}  \\
 x_{k,i}&= \sum_{s \in \cN_k} a_{sk} z_{s,i} \quad &\textbf{(Communication step)}    \\
w_{k,i} &= {\bf prox}_{\mu R}(x_{k,i}) 
}
\end{subequations}
\end{algorithm}
\subsection{Prox-ATC II: $\bar{\cA}=\cA$, $\cB=I-\cA$, and $\cC=I-\cA$}
 For the choice $\bar{\cA}=\cA$, $\cB=I-\cA$, and $\cC=I-\cA$, we can represent \eqref{alg_prox_ATC_framework} as listed in \eqref{prox_ATCII}. This can be seen by following the same approach as the previous subsection. To see this, note that with $\cB^2= (I- \cA)^2$ to get
  \eq{
 \ssz_i &=(2\cA  -\cA^2) \ssz_{i-1}  
+ \cA \sw_{i-1} - \cA  \sw_{i-2} -\mu \big(\grad \cJ(\sw_{i-1})-\grad \cJ(\sw_{i-2})\big)
 }
 By multiplying $\cA$ to both sides of the previous equation and introducing $\ssx_i \define \cA \ssz_i$ we get
  \eq{
 \ssx_i&=\cA \bigg( (2I-\cA) \ssx_{i-1}+ \cA \sw_{i-1}-\cA \sw_{i-2}-\mu  \big(\grad \cJ(\sw_{i-1})-\grad \cJ(\sw_{i-2})\big) \bigg)
 }
 Thus from \eqref{primal_prox_ATC_DIG} we have $
\sw_i={\rm \bf prox}_{\mu R}( \ssx_i)$. 
\begin{algorithm}[h] 
\caption*{\textrm{\bf{Algorithm}}: Prox-ATC II}
{\bf Setting}:  Initialize $x_{k,-1}=\psi_{k,-1}=0$ and $w_{k,-1}$ arbitrary. For every agent $k$, repeat for $i=0,1,2,...$
\begin{subequations} \label{prox_ATCII}
\eq{
  \psi_{k,i}&=2x_{k,i-1} -\mu \big( \grad J_k(w_{k,i-1})-\grad J_k(w_{k,i-2})  \big)   \\
  z_{k,i}&=  \psi_{k,i} -\sum_{s \in \cN_k}a_{sk} ( x_{s,i-1}-w_{s,i-1}+w_{s,i-2} )  \quad &\textbf{(Communication step)} \\ 
x_{k,i}&= \sum_{s \in \cN_k} a_{sk} z_{s,i} \quad &\textbf{(Communication step)}    \\
w_{k,i} &= {\bf prox}_{\mu R}(x_{k,i}) 
}
\end{subequations}
\end{algorithm}
\section{Proximal mapping of \eqref{Rk}}
\label{app_counter_example_proximal}
 To rewrite the non-smooth terms \eqref{Rk} more compactly, we introduce 
\eq{
	D_1 & \define
	\ba{ccccccccccc}
	\sqrt{2} & 0 & 0 & 0 & 0 & 0 & 0& \cdots & 0 & 0 & 0 \\
	0 & 1 & -1 & 0 & 0 & 0 & 0& \cdots & 0 & 0 & 0 \\
	0 & 0 & 0 & 1 & -1 & 0 & 0& \cdots & 0 & 0 & 0 \\
	\vdots & \vdots & \vdots & \vdots & \vdots & \vdots & \vdots & \cdots & \vdots & \vdots & \vdots \\
	0 & 0 & 0 & 0 & 0 & 0 & 0& \cdots & 1 & -1 & 0 
	\ea	\in \real^{\frac{M}{2}\times M} \label{D1} \\
	D_2 & \define
	\ba{ccccccccccc}
	1 & -1 & 0 & 0 & 0 & 0 & 0& \cdots & 0 & 0 & 0 \\
	0 & 0 & 1 & -1 & 0 & 0 & 0& \cdots & 0 & 0 & 0 \\
	0 & 0 & 0 & 0 & 1 & -1 & 0& \cdots & 0 & 0 & 0 \\
	\vdots & \vdots & \vdots & \vdots & \vdots & \vdots & \vdots & \cdots & \vdots & \vdots & \vdots \\
	0 & 0 & 0 & 0 & 0 & 0 & 0& \cdots & 0 & 1 & -1 
	\ea	\in \real^{\frac{M}{2}\times M} \label{D2}
}
and $b_1 \define e_1$ where $e_1$ is the first column of the identity matrix $I_{M/2}$. With $D_1$, $D_2$ and $b_1$, we can rewrite $R_1(w)$ and $R_2(w)$ in \eqref{Rk} as
\eq{
	R_1(w) = \|D_1 w - b_1\|_1, \quad R_2(w) = \|D_2 w\|_1.
}
 Let us  introduce $g(w) = \|w\|_1$ so that $R_1(w) = g(D_1 w - b_1)$ and $R_2(w) = g(D_2 w)$. It can be verified that $D_1 D_1\tran = 2 I$ and $D_2 D_2\tran = 2 I$. Thus, from \cite[Theorem~6.15]{beck2017first} it holds that
 \begin{subequations}
\eq{
	{\bf prox}_{\mu R_1}(w) &= w + \frac{1}{2\mu}D_1\tran[{\bf prox}_{2\mu^2 g}(\mu D_1 w - \mu b_1) - \mu D_1 w + \mu b_1], \\
	{\bf prox}_{\mu R_2}(w) &= w + \frac{1}{2\mu}D_2\tran[{\bf prox}_{2\mu^2 g}(\mu D_2 w) - \mu D_2 w].
}
\end{subequations}
In other words, both ${\bf prox}_{\mu R_1}(w)$ and ${\bf prox}_{\mu R_2}(w)$ have closed forms which are easy to calculate since  
${\bf prox}_{\kappa g}(w) = \col \left\{ \sgn(w[j])\max\{\big|w[j]\big| - \kappa ,0\} \right\}_{j=1}^M \in \real^M$.
\section{NON-ATC Algorithms ($\bar{\cA}=I$)}
\label{supp_non_atc}
Consider the special case of \eqref{alg_ATC_framework} with $\bar{\cA}=I$:
\begin{subnumcases}{\label{alg_non_ATC_framework}}
\sw_i  =  (I-\cC) \sw_{i-1}-\mu \grad \cJ(\sw_{i-1})  -  \cB \sy_{i-1} \label{z_nonATC_DIG}  &\textbf{(primal-descent)} \\
\sy_i = \sy_{i-1}+ \cB  \sw_i \label{dual_nonATC_DIG}  &\textbf{(dual-ascent)} 
 \end{subnumcases}
 In this section, we will analyze \eqref{alg_non_ATC_framework} under the slightly different conditions. This is because the assumption imposed in \eqref{assump_combination} is not satisfied for the non-ATC case $\bar{\cA}=I$. We remark that we can also study the non-smooth recursion \eqref{alg_prox_ATC_framework} with $\bar{\cA}=I$ by adjusting the technique from \cite{alghunaim2019linearly}, which analyzed a specific {\em NON-ATC} instance of \eqref{alg_prox_ATC_framework} with $\bar{\cA}=I$. However, it would require stricter step-size conditions due to complication of the proximal term. Therefore, we will focus on the smooth case $R(w)=0$ to get wider step-size conditions. 
 
 We begin by showing that \eqref{alg_non_ATC_framework} covers DIGing \cite{nedic2017achieving}, EXTRA \cite{shi2015extra},  and DLM \cite{ling2015dlm}. Similar to the main paper, with $\sy_0=0$, we can eliminate the dual variable to get the equivalent representation:
\eq{\label{decentralized_implem_sup}
\sw_i &= (2I -    \cC-\cB^2) \sw_{i-1} - (I-\cC )  \sw_{i-2} \hspace{-0.5mm}-\hspace{-0.5mm} \mu \big(\grad \cJ(\sw_{i-1}) \hspace{-0.5mm}-\hspace{-0.5mm} \grad \cJ(\sw_{i-2})\big)
}
The above algorithm can cover DIGing \cite{nedic2017achieving}, EXTRA \cite{shi2015extra},  and DLM \cite{ling2015dlm} as special cases:
 \begin{itemize}
 \item (DIGing \cite{nedic2017achieving}): If $\cB^2=(I-\cA)^2$ and $\cC=I-\cA^2$, then we recover the DIGing form given in \cite[Section 2.2.1]{nedic2017achieving}:
\eq{\label{diging_alg}
\sw_i &= 2\cA \sw_{i-1} -\cA^2 \sw_{i-2} \hspace{-0.5mm}-\hspace{-0.5mm} \mu \big(\grad \cJ(\sw_{i-1}) \hspace{-0.5mm}-\hspace{-0.5mm} \grad \cJ(\sw_{i-2})\big)
}
\item (EXTRA \cite{shi2015extra}): If $\cB^2=0.5(I-\cA)$ and $\cC=0.5(I-\cA)$, then we recover EXTRA:
\eq{\label{extra_alg}
\sw_i = 0.5(I+\cA) (2\sw_{i-1} -  \sw_{i-2}) \hspace{-0.5mm}-\hspace{-0.5mm} \mu \big(\grad \cJ(\sw_{i-1}) \hspace{-0.5mm}-\hspace{-0.5mm} \grad \cJ(\sw_{i-2}) \big)
}
\item (DLM from \cite{ling2015dlm}) Consider an instance\footnote{We let $\tilde{d}_k=2cd_k+\rho={1 \over \mu}$ in the DLM from \cite{ling2015dlm}. } of the decentralized linearized ADMM (DLM) method from \cite{ling2015dlm}:
 \begin{subequations} \label{lin_admm}
\eq{
\sw_i&=\sw_{i-1}-\mu \big( \grad_{\ssw} \cJ(\sw_{i-1})+c \cL \sw_{i-1}+ \sy_{i-1} \big) \label{admm1} \\
\sy_i &= \sy_{i-1}+c \cL \sw_i \label{admm2}
}
 \end{subequations}
 where $\mu,c>0$. The matrix $\cL$ is the matrix chosen such that the $k$-th block of $\cL \sw_i$ is equal to $\sum_{s \in \cN_k} w_{k,i}-w_{s,i}$.  Eliminating the dual variable from \eqref{lin_admm}, we get:
\eq{
\sw_i = (I-\mu c \cL) (2\sw_{i-1} -  \sw_{i-2}) \hspace{-0.5mm}-\hspace{-0.5mm} \mu \big(\grad \cJ(\sw_{i-1}) \hspace{-0.5mm}-\hspace{-0.5mm} \grad \cJ(\sw_{i-2})\big) \label{DLM}
} 
which is equivalent to \eqref{decentralized_implem_sup} with $\cB^2=\mu c \cL$ and $\cC=\mu c \cL$. Notice that DLM \eqref{DLM} and EXTRA \eqref{extra_alg} have the same form and differ only by the choice of matrices multiplying the term $(2\sw_{i-1} -  \sw_{i-2})$.
\end{itemize} 
For the analysis of the NON-ATC form, we impose the following assumption.
\begin{assumption} \label{assum_supp_com}
{\rm We assume that
\eq{ 
	 \cC \sw =0 \iff \cB \sw=0   \iff w_1=\cdots=w_K , 
	 } 
	and the following condition hold:
\eq{
0 \leq \cB^2 \leq \cC < I \label{assum_supp_combination}
}
}
\qd
\end{assumption} 
We note that the above assumption is consistent with the assumption used to analyze EXTRA \cite{shi2015extra}. Specifically, the EXTRA case ($\cB^2=0.5(I-\cA)$ and $\cC=0.5(I-\cA)$) satisfies \eqref{assum_supp_combination} for any primitive, symmetric and doubly stochastic $\cA$. It is also satisfied by DLM ($\cB^2=\mu c \cL$ and $\cC=\mu c \cL$) since we can always choose small enough $c$ or $\mu$. For the DIGing case, condition \eqref{assum_supp_combination} implies that $0<\cA\leq 1$. Although not necessary, it can be easily satisfied and allow us to derive tighter steps-size upper bounds -- see Remark \eqref{remark:conv_conditions}.

 Similar to the main body in the paper, we let $(\sw^\star,\sy_b^\star)$ to be the particular saddle-point where $\sy_b^\star$ is the unique vector in the range space of $\cB$. Then we know from Lemma \ref{lemma:existence_fixed_optimality} that this point coincide with the fixed point of \eqref{alg_non_ATC_framework} and satisfies the optimality conditions:
\begin{subnumcases}{}
	\mu \grad \cJ(\sw^\star) + \cB \sy_{b}^\star=0 \label{fixed_primal} \\
  \cB \sw^\star=0 \label{fixed_dual} 
\end{subnumcases}
Note that $\cC \sw^\star=0$. Thus, using the above fixed point we can get the following error recursion dynamics:
\begin{subnumcases}{}
\tsw_i=\tsw_{i-1}-\mu \big(\grad \cJ(\sw_{i-1})-\grad \cJ(\sw^\star) \big)- \cC \tsw_{i-1} -\cB \tsy_{i-1} \label{error_primal_supp} \\
\tsy_i = \tsy_{i-1}+\cB \tsw_i \label{error_dual_supp} 
\end{subnumcases}
where $\tsw_i=\sw_i-\sw^\star$ and $\tsy_i=\sy_i-\sy_b^\star$. To simplify the notation in the analysis of the next results, we define $\sigma_{\max}\define \sigma_{\max}(\cC)<1$ and $0< \underline{\sigma} \define \underline{\sigma}(\cB^2)<1$.
\begin{lemma}[\sc Descent inequality]
{\rm  
		Under Assumptions \ref{assump:cost} and \ref{assum_supp_com} and the step-size condition 
		\eq{
		\mu < {2 \big(1-\sigma_{\max}(\cC)\big) \over \delta}
		}
		 it holds that 
		\eq{
				\|\tsw_{i-1}-\mu \big(\grad \cJ(\sw_{i-1})-\grad \cJ(\sw^\star)\big)- \cC \tsw_{i-1}\|^2
	\leq \ \left(1-  \mu \nu \big(2-\tfrac{\mu  \delta}{ 1-\sigma_{\max}(\cC)}\big)\right) \|\tsw_{i-1}\|^2   -   \| \tsw_{i-1}\|_{\cC}^2  
			\label{bound_descent_0_supp}
		}
	} 
\end{lemma} 
\begin{proof} It holds that:
\eq{
	&\|\tsw_{i-1}-\mu \big(\grad \cJ(\sw_{i-1})-\grad \cJ(\sw^\star)\big)- \cC \tsw_{i-1}\|^2  \nonumber \\
	&  = \|\tsw_{i-1}\|^2  \hspace{-0.5mm}-\hspace{-0.5mm} 2 \mu \tsw_{i-1}\tran \big(\grad \cJ(\sw_{i-1}) \hspace{-0.5mm}-\hspace{-0.5mm} \grad \cJ(\sw^\star) \big) \hspace{-0.5mm} - \hspace{-0.5mm} 2\| \tsw_{i-1}\|_{\cC}^2 \hspace{-0.5mm} \nonumber \\
	& \quad + \| \mu \grad \cJ(\sw_{i-1})-\mu \grad \cJ(\sw^\star)+ \cC \tsw_{i-1} \|^2   
	\label{squared_primal_descent_supp}}
From Jensen's inequality, it holds for any $t \in (0,1)$ that
\eq{
	 \| \mu \grad \cJ(\sw_{i-1})-\mu \grad \cJ(\sw^\star)+ \cC \tsw_{i-1} \|^2  
	& \leq {\mu^2 \over 1-t}\| \grad \cJ(\sw_{i-1})-\grad \cJ(\sw^\star)\|^2 +{1 \over t }\|   \tsw_{i-1} \|_{\cC^2}^2 \nonumber \\
	&\leq \tfrac{\mu^2}{ 1-\sigma_{\max}}\| \grad \cJ(\sw_{i-1})-\grad \cJ(\sw^\star)\|^2 + \|    \tsw_{i-1} \|_{\cC}^2 \label{bound_jensen_supp}
}
where in the last step we set $t=\sigma_{\max}<1$ and used the upper bound $\|  \tsw_{i-1} \|^2_{\cC^2} \leq \sigma_{\max}\|   \tsw_{i-1} \|_{\cC}^2$. Also, since $\cJ(\sw)$ is  $\delta$-smooth, it holds that \cite[Theorem 2.1.5]{nesterov2013introductory}:
\eq{
\| \grad \cJ(\sw_{i-1})-\grad \cJ(\sw^\star)  \|^2 \leq  \delta \tsw_{i-1}\tran  \big(\grad \cJ(\sw_{i-1})-\grad \cJ(\sw^\star)  \big)
}
Substituting the above two bound into \eqref{squared_primal_descent_supp} we get
\eq{
&	\|\tsw_{i-1}-\mu \big(\grad \cJ(\sw_{i-1})-\grad \cJ(\sw^\star)\big)- \cC \tsw_{i-1}\|^2
	 \nonumber \\
	 &\leq \  \|\tsw_{i-1}\|^2 -  \mu \left(2- \tfrac{\mu \delta}{ 1-\sigma_{\max}}\right) \tsw_{i-1}\tran \big(\grad \cJ(\sw_{i-1})   -\grad \cJ(\sw^\star)  \big)-   \| \tsw_{i-1}\|_{\cC}^2  \nonumber \\
	 &\leq \ \left(1-  \mu \nu \left(2- \tfrac{\mu \delta}{ 1-\sigma_{\max}}\right)\right) \|\tsw_{i-1}\|^2   -   \| \tsw_{i-1}\|_{\cC}^2 
}
where the last step holds for $\mu < {2(1-\sigma_{\max}) \over \delta}$ due to the strong-convexity condition \eqref{stron-convexity}. 
\end{proof}
\begin{theorem}[Linear convergence] \label{theorem_extra_supp} {\rm Under Assumptions \ref{assump:cost} and \ref{assum_supp_com} and if the step-size satisfies
	\eq{\label{step-size-cond-2_supp}
		\mu < \frac{2\big(1-\sigma_{\max}(\cC)\big)}{\delta}
	}
	It holds that $\|\tsw_i\|^2_{\cQ} \le C\rho^i$ where $\cQ>0$, $
		\gamma 	\define \max\left\{ 1-  \mu \nu \big(2- \tfrac{ \mu \delta}{ 1-\sigma_{\max}(\cC)}\big), 1 -  \underline{\sigma}(\cB^2) \right\} < 1$ and $C\geq 0$.	
		}
\end{theorem}
\begin{proof}
Squaring both sides of \eqref{error_primal_supp} and \eqref{error_dual_supp} we get
\eq{
	\|\tsw_i\|^2&= \|\tsw_{i-1}-\mu \big(\grad \cJ(\sw_{i-1})-\grad \cJ(\sw^\star)\big)- \cC \tsw_{i-1}\|^2 +  \| \cB \tsy_{i-1}\|^2 \nonumber \\
	& \ -2  \tsy_{i-1}\tran \cB \left(\tsw_{i-1}-\mu \big(\grad \cJ(\sw_{i-1})-\grad \cJ(\sw^\star)\big)- \cC \tsw_{i-1}\right) 
	\label{er_sq_primal_supp}
}
and
\eq{
	\|\tsy_i\|^2  =\|\tsy_{i-1}+ \cB \tsw_i \|^2 &= \|\tsy_{i-1}\|^2+ \| \cB \tsw_i \|^2 + 2  \tsy_{i-1} \tran \cB \tsw_i \nonumber \\
	&\overset{\eqref{error_primal_supp}}{=} \|\tsy_{i-1}\|^2+ \| \tsw_i \|^2_{\cB^2} - 2    \|\cB \tsy_{i-1}\|^2 \nonumber \\ 
	& \quad +2   \tsy_{i-1}\tran \cB  \left(\tsw_{i-1}-\mu \big(\grad \cJ(\sw_{i-1})-\grad \cJ(\sw^\star)\big)- \cC \tsw_{i-1}\right) \label{er_sq_dual_supp}
}
Adding \eqref{er_sq_dual_supp} to \eqref{er_sq_primal_supp}, we get 
\eq{
\|\tsw_i\|^2_{\cQ} \hspace{-0.6mm}+\hspace{-0.6mm} \|\tsy_i\|^2 \hspace{-0.6mm}=\hspace{-0.6mm} \|\tsw_{i-1}-\mu \big(\grad \cJ(\sw_{i-1})-\grad \cJ(\sw^\star)\big)- \cC \tsw_{i-1}\|^2 \hspace{-0.6mm}+\hspace{-0.6mm} \|\tsy_{i-1}\|^2 \hspace{-0.6mm}-\hspace{-0.6mm}      \|\cB \tsy_{i-1}\|^2 \label{err_sum_supp}
}
where $\cQ = I - \cB^2>0$ since $\cB^2<I$. Since  both $\sy_i$ and $\sy_b^\star$ lie in the range space of $\cC$, it  holds that $
\|\cB \tsy_{i-1}\|^2 \geq 
\underline{\sigma} \|\tsy_{i-1}\|^2 $. Thus, we can bound \eqref{err_sum_supp} by
\eq{
	\|\tsw_i\|^2_{\cQ}+ \|\tsy_i\|^2 
	& \le  \|\tsw_{i-1}-\mu \big(\grad \cJ(\sw_{i-1})-\grad \cJ(\sw^\star)\big)- \cC \tsw_{i-1}\|^2 \hspace{-0.5mm}+\hspace{-0.5mm} (1- \underline{\sigma})\|\tsy_{i-1}\|^2 \label{err_sum1_supp}
}
When $\mu \le \frac{2(1 - \sigma_{\max})}{\delta }$, we can substitute  inequality \eqref{bound_descent_0_supp} into the above relation and get
\eq{
	\|\tsw_i\|^2_{\cQ} \hspace{-0.5mm}+\hspace{-0.5mm} \|\tsy_i\|^2 
	& \le \left(1-  \mu \nu \big(2-\tfrac{\mu  \delta}{(1-\sigma_{\max})}\big)\right)\hspace{-0.5mm}\|\tsw_{i-1}\|^2 \hspace{-0.5mm}-\hspace{-0.5mm}   \| \tsw_{i-1}\|_{\cC}^2 \hspace{-0.5mm}+\hspace{-0.5mm} (1- \underline{\sigma})\|\tsy_{i-1}\|^2 \label{err_sum1-2_supp}
}
Let $\gamma_1 \define 1-  \mu \nu \big(2- { \mu \delta \over 1-\sigma_{\max}}\big) $, $\gamma_2 \define 1- \underline{\sigma}$, and add and subtract $\gamma_1 \| \sw_{i-1}\|_{\cB^2}^2$  then the above inequality can be rewritten as
\eq{
	\|\tsw_i\|^2_{\cQ} \hspace{-0.5mm}+\hspace{-0.5mm} \|\tsy_i\|^2 
	& \le \gamma_1 \hspace{-0.5mm}\|\tsw_{i-1}\|^2_{\cQ} \hspace{-0.5mm} + \gamma_1  \| \sw_{i-1}\|_{\cB^2}^2 -\| \sw_{i-1}\|_{\cC}^2 +\hspace{-0.5mm} \gamma_2 \|\tsy_{i-1}\|^2 \label{err_sum1-3_supp}
}
Note that $\gamma_1<1$ if $\mu < {2 (1-\sigma_{\max}) \over \delta}$. Therefore, using condition \eqref{assum_supp_combination}, we have:
\eq{
\gamma_1  \| \sw_{i-1}\|_{\cB^2}^2 -\| \sw_{i-1}\|_{\cC}^2 \leq -(1-\gamma_1)  \| \sw_{i-1}\|_{\cC}^2  \leq 0
}
Using the previous bound and  letting $\gamma = \max\{\gamma_1, \gamma_2\}$, inequality \eqref{err_sum1-3_supp} is upper bounded by 
	\begin{align}
		\|\tsw_i\|^2_{\cQ} + \|\tsy_i\|^2 
		&\leq \gamma \left( \|\tsw_{i-1}\|^2_{\cQ}  + \|\tsy_{i-1}\|^2 \right).
	\end{align}
	Since $\cQ = I -  \cB^2$ is positive definite, we reach the linear convergence of $\tsw_i$. 
	\end{proof}
\end{document}